\documentclass[11pt]{article}

\usepackage{amssymb,amsmath,bm}
\usepackage{amsthm}
\usepackage{empheq}
\usepackage{textcomp}
\usepackage{enumerate}
\usepackage{graphicx}
\usepackage{caption}
\usepackage{nicematrix}
\usepackage{tikz-cd}

\usepackage{mathrsfs}

\usepackage{url}

\renewcommand{\setminus}{{\smallsetminus}}

\usepackage{amssymb,amsmath,bm}
\usepackage{amsthm}
\usepackage{hyperref}
\usepackage{mathrsfs}
\usepackage{textcomp}
\usepackage{enumerate}
\usepackage{graphicx}
\usepackage{tikz}
\usepackage{verbatim}

\newtheorem{theorem}{Theorem}[section]
\newtheorem{lemma}[theorem]{Lemma}
\newtheorem{proposition}[theorem]{Proposition}

\newtheorem{corollary}[theorem]{Corollary}
\newtheorem{conjecture}[theorem]{Conjecture}

\theoremstyle{remark}

\theoremstyle{remark}

\numberwithin{equation}{section}

\newcommand{\lt}{\left}
\newcommand{\rt}{\right}

\def \CC{\mathbb C}

\def\CC{{\mathbb C}}
\def\RR{{\mathbb R}}
\def\ZZ{{\mathbb Z}}

\def \det{\operatorname{det}}

\title{Generalized Bonahon-Wong-Yang volume conjecture of\\ quantum invariants of surface diffeomorphisms I: \\the figure eight knot complement}
\author{Tushar Pandey and Ka Ho Wong}
\date{}

\begin{document}

\maketitle

\begin{abstract}
We propose a generalization of the Bonahon-Wong-Yang volume conjecture of quantum invariants of surface diffeomorphisms, by relating the asymptotics of the invariants with certain hyperbolic cone structure on the mapping torus determined by the choice of the invariant puncture weights. We prove the conjecture for the once-punctured torus bundle with the diffeomorphism given by the word $LR$.
\end{abstract}

\tableofcontents

\section{Introduction}
In \cite{BWYI}, Bonahon-Wong-Yang define a sequence of real-valued invariants for self-diffeomorphisms of surfaces by using the representation theory of the Kauffman bracket skein algebras. To be precise, for each positive odd integer $n\geq 3$, consider the primitive root of unity $A=e^{\pi\sqrt{-1}/n}$. Let $\Sigma=\Sigma_{g,p}$ be an oriented surface with $g$ genus and $p\geq 1$ punctures. Let $\mathcal{K}^q(\Sigma)$ be the Kauffman bracket skein algebra of $\Sigma$ at $q=A^2=e^{2\pi\sqrt{-1}/n}$. The results in \cite{BW1, FKL, GJS19} provide a bijection between the set of isomorphism classes of irreducible finite dimensional representations of the skein algebra $\mathcal{K}^q(\Sigma)$ and the set of pairs $(r,\{p_v\})$, where $r$ is an element (called \emph{the classical shadow}) in the $\mathrm{SL}(2;\mathbb{C})$-character variety of $\Sigma$ and $\{p_v\}$ is a finite set of complex numbers (called \emph{the puncture weights}) associated to each puncture of $\Sigma$. For each puncture $v$ with a small loop $\mu_v$ around $v$, the puncture weight can be regarded as a choice of ``$n$-th root" of $\mathrm{Trace}(r([\mu_v])$ in the sense that if $\mathrm{Trace}(r([\mu_v]) = -e^{h_v}-e^{-h_v}$ for some $h_v\in \CC$, then $p_v = e^{(h_v+2\pi m_v\sqrt{-1})/n} + e^{-(h_v+2\pi m_v\sqrt{-1})/n}$ for some $m_v\in \mathbb{Z}$. The choice of puncture weight will play an important role in this paper to understand the connection between the asymptotics of the invariants and the hyperbolic cone structure on the mapping torus.

For any pseudo-Anosov diffeomorphism $\varphi: \Sigma \to \Sigma$, it is well-known that the mapping torus $M_\varphi$ obtained from $\Sigma \times [0,1]$ by gluing $(x,1) \sim (\varphi(x),0)$ is hyperbolic \cite{T88}. In particular, the restriction of the holonomy representation of $M_\varphi$ on $\pi_1(\Sigma)$ induces a $\mathrm{SL}(2;\CC)$-character of $\Sigma$ that is invariant under the action of $\varphi$ \cite[Section 3.2]{BWYI}. In this paper, we always assume that such a $\varphi$-invariant character $r$ is a generic character in the smooth part of the $\mathrm{SL}(2;\CC)$-character variety of $\Sigma$ (see Section \ref{EFQISD} for more details). By \cite{BW1, BW2, BW3}, for each choice of puncture weights $\{p_v\}$ satisfying $p_v = p_{\varphi(v)}$, the pairs $(r, \{p_v\})$ and $([r\circ \varphi], \{p_{\varphi(v)}\})$ correspond to two isomorphic irreducible representations of the Kauffman bracket skein algebra of $\Sigma$ to some vector space $V$ that are related by an intertwiner $\Lambda^q_{\varphi, r,p_v} : V\to V$. In  \cite{BWYI}, Bonahon, Wong and Yang prove that the modulus of the trace of the intertwiner $\Lambda^q_{\varphi, r,p_v}$ is real-valued invariant that depends only on the diffeomorphism $\varphi$, the $\varphi$-invariant character $r$, the root of unity $q$ and the $\varphi$-invariant puncture weight $p_v$. Furthermore, they propose the following volume conjecture of the quantum invariants of surface diffeomorphisms.
\begin{conjecture}[Conjecture 6, \cite{BWYI}]\label{BWYVC}
For the sequence of puncture weights $p_v = e^{\frac{h_v}{n}} + e^{-\frac{h_v}{n}}$,
$$\lim_{\substack{ n\to \infty \\n \text{ odd}}} \frac{4\pi}{n} \ln|\mathrm{Trace}(\Lambda^q_{\varphi, r,p_v})| = \mathrm{Vol}(M_\varphi),$$
where $\mathrm{Vol}(M_\varphi)$ is the hyperbolic volume of the mapping torus $M_\varphi$ with the complete hyperbolic structure.
\end{conjecture}

Bonahon, Wong and Yang prove Conjecture \ref{BWYVC} for the mapping torus that is homeomorphic to the figure eight knot complement \cite{BWYII} and all once-punctured torus bundles under certain assumptions of the volume \cite{BWYIII}. Besides, the first author proves Conjecture \ref{BWYVC} for four-punctured sphere bundles under similar assumptions \cite{PT}. A special feature of Conjecture \ref{BWYVC} is that the asymptotics of the quantum invariants always capture the volume of the mapping torus with the complete hyperbolic structure for different $\varphi$-invariant characters r with different representation volume.

The main goal of this paper is to investigate the possibility of these quantum invariants to capture the geometric information of the incomplete hyperbolic structure on the mapping torus $M_{\varphi}$. The main observation is that the limits of the puncture weights play a role in the asymptotics of the invariants and are related to hyperbolic cone metrics on mapping torus determined by those limits. More precisely, for each puncture $v$, let $m_{v}^{(n)}\in \{-\frac{n-1}{2}, - \frac{n-3}{2},\dots, \frac{n-1}{2}\}$ be a sequence of integers such that the limit $\lim_{n\to \infty} \frac{4\pi m_{v}^{(n)}}{n}$ exists, and let
$$ \theta_v = \Bigg|\lim_{n\to \infty} \frac{4\pi m_{v}^{(n)}}{n}\Bigg|  \in [0, 2\pi].$$ 
Assume that there exists a hyperbolic cone metric on the mapping torus $M_\varphi$ with cone angle $\theta_v$ around the meridian $\mu_v$ of each puncture $v$. Denote such a hyperbolic cone manifold by $M_{\varphi,\theta_v}$. We propose the following generalization of Conjecture \ref{BWYVC}.

 \begin{conjecture}\label{PWVC}
For the sequence of puncture weights $p_v  = e^{\frac{h+2\pi m_v^{(n)}\sqrt{-1}}{n}} + e^{-\frac{h+2\pi m_v^{(n)}\sqrt{-1}}{n}}$,
$$\lim_{\substack{ n\to \infty \\n \text{ odd}}} \frac{4\pi}{n} \ln|\mathrm{Trace}(\Lambda^q_{\varphi, r,p_v})| = \mathrm{Vol}(M_{\varphi,\theta_v}),$$
where $\mathrm{Vol}(M_{\varphi,\theta_v})$ is the hyperbolic volume of the cone manifold $M_{\varphi,\theta_v}$.
\end{conjecture}

In particular, Conjecture \ref{BWYVC} is a special case of Conjecture \ref{PWVC} where $m_v^{(n)}=0$ for any $n$. Besides, Conjecture \ref{PWVC} should be compared with volume conjecture of the relative Reshetikhin-Turaev invariants (see Conjecture 1.1 in \cite{WY1}). The sequence of puncture weights $m_v^{(n)}$ plays essentially the same role as the sequence of colorings of the framed link in the definition of the relative Reshetikhin-Turaev invariants.

In this paper, we study the asymptotics of the once-punctured torus bundle $\Sigma=\Sigma_{1,1}$ with pseudo-Anosov diffeomorphism $\varphi=LR$. Topologically the mapping torus is homeomorphic to the figure eight knot complement, since they share the same ideal triangulation \cite{T, Gue}. Nevertheless, it is worth mentioning that the meridian of $\Sigma$ corresponds to the longitude of the figure eight knot since both of them bound a Seifert surface inside the manifolds. For the rest of the paper, meridian means the meridian of the surface $\Sigma$ and not the figure eight knot. It is known that for any cone angle $\theta_v \in [0,2\pi)$, there exists a hyperbolic cone structure on $\Sigma$ with cone angle $\theta$ along the meridian \cite{HLM}. We provide another proof of this result in Lemma \ref{geotricone} by constructing a geometric triangulation of the cone manifold $M_{\varphi,\theta_v}$ for any cone angle $\theta_v \in [0,2\pi)$.

The main result of this paper is Theorem \ref{mainthm}. Let $\Sigma=\Sigma_{1,1}$ be the once-punctured torus and let $\varphi = LR$. Let
$m_{v}^{(n)}\in \{-\frac{n-1}{2}, - \frac{n-3}{2},\dots, \frac{n-1}{2}\}$ be a sequence of integers such that $\lim_{n\to \infty} \frac{4\pi m_{v}^{(n)}}{n}$ exists. We further assume that 
$$ \theta_v = \Bigg| \lim_{n\to \infty} \frac{4\pi m_{v}^{(n)}}{n}\Bigg|  \in [0, 2\pi).$$
 Let $\theta_v^{(n)} = \frac{4\pi m_{v}^{(n)}}{n}$ and let $\rho_{\theta_v^{(n)}} : \pi_1(\Sigma) \to \mathrm{PSL}(2;\CC)$ be the holonomy representation of the cone structure on $\Sigma$ with cone angle $\theta_v^{(n)}$. Let $\mathrm{Vol}\big(M_{\varphi, \theta_v^{(n)}}\big)$ be volume of $M_{\varphi}$ with the cone structure given by $\rho_{\theta_v^{(n)}}$. Let $\mathrm{Tor}(M_\varphi, \mu_v,  \rho_{\theta_v^{(n)}})$ be the adjoint twisted Reidemeister torsion of the mapping torus $M_\varphi $ with respect to the meridian $\mu_v$ and the representation $\rho_{\theta_v^{(n)}}$.
\begin{theorem}\label{mainthm}
The asymptotic expansion formula of $|\mathrm{Trace} \Lambda_{\varphi, r, p_v}^q |$ is given by
\begin{align*}
\Big|\mathrm{Trace}\text{ }  (\Lambda_{\varphi, r, p_v}^q) \Big|
= \frac{C(n) }{4\pi} \frac{e^{\frac{n}{4\pi} \mathrm{Vol}\big(M_{\varphi, \theta_v^{(n)}}\big)}}{\Big|\sqrt{\mathrm{Tor}(M_\varphi, \mu_v,  \rho_{\theta_v^{(n)}})}\Big|} \Big(1+O\Big(\frac{1}{n}\Big)\Big),
\end{align*}
where $C(n)$ is a sequence of real numbers satisfying $A< C(n) < B$ for some constants $A,B>0$.
\end{theorem}

The precise formula of $C(n)$ can be found in Proposition \ref{lastpf}.
Theorem \ref{mainthm} should be compared with the asymptotic expansion conjecture of the relative Reshetikhin-Turaev invariants \cite{WY4}, which suggest that the adjoint twisted Reidemeister torsion appears as the 1-loop term in the asymptotic expansion formula of the quantum invariants. As an immediate consequence of Theorem \ref{mainthm}, we have
\begin{corollary}
Conjecture \ref{PWVC} holds for the once-punctured torus bundle with pseudo-Anosov homeomorphism $\varphi=LR$ with any cone angle in $[0,2\pi)$.
\end{corollary}

Finally, it is natural to ask whether Conjectures \ref{BWYVC} and \ref{PWVC} can be complexified by relating the asymptotics of $\mathrm{Trace}\text{ }  (\Lambda_{\varphi, r, p_v}^q) $ with the complex volume of the cone manifold $M_{\varphi, \theta_v}$. One difficulty is that $\mathrm{Trace}\text{ }  (\Lambda_{\varphi, r, p_v}^q) $ is not an invariant. Nevertheless, we observe that the Neumann-Zagier potential function naturally shows up when we study the asymptotics of the invariant. See Section \ref{NZPF} for more details.

\section*{Plan of this paper}
In Section \ref{Prelim}, we recall several basic ingredients for the definition and estimation of the quantum invariants. In Section \ref{EFQISD}, we find an explicit formula of the quantum invariants and rewrite the invariants as a sum of values of a certain holomorphic function defined by using the quantum dilogarithm function. By using the Poisson Summation Formula, we rewrite the invariants in terms of the Fourier coefficients together with some error terms. In Section \ref{GPF}, we discuss how the potential function captures the geometric information of the mapping torus. In Section \ref{ALFC}, by using the saddle point approximation, we obtain the asymptotics of the leading Fourier coefficients. Furthermore, in Section \ref{EOFC}, we show that all the other Fourier coefficients are negligible. Finally, we prove Theorem \ref{mainthm} in Section \ref{AOI}.

\section*{Acknowledgement}
The authors would like to thank Tian Yang and Francis Bonahon for their valuable discussions. Part of this research was completed when the first author visited Yale University. The authors are grateful to the Mathematics department of Yale University for the hospitality. The first author is supported by the NSF grants DMS-1812008 and DMS-2203334 (PI: Tian Yang).

\section{Preliminary}\label{Prelim}
Let $q=e^{\frac{2\pi\sqrt{-1}}{n}}$, where $n$ is odd. The materials in this section can be found in \cite{BWYI, BWYII}.
\subsection{Discrete quantum dilogarithm function}
Given two complex numbers $u, v \in \CC$ satisfying $v^n = 1+u^n \neq 0$, the \emph{discrete quantum dilogarithm function} \cite{FK94}, denoted by $\mathrm{QDL}^q(u,v |j)$, is defined by
$$   \mathrm{QDL}^q(u,v |i) = v^{-i} \prod_{k=1}^i \big(1+uq^{-2k}\big). $$

\begin{lemma}[Lemma 22, \cite{BWYI}] For any $j=0,\dots,n-1$,
$$   \mathrm{QDL}^q(u,v | j+n) = \mathrm{QDL}^q(u,v | j) . $$
In particular, $\mathrm{QDL}^q(u,v | j)$ can be defined for all $i \in \ZZ$.
\end{lemma}

\begin{proposition}[Lemma 7, \cite{BWYII}] \label{Estbridge}
Given $U, V \in \CC$ with $e^V = 1+e^U$, choose $q=e^{\frac{2\pi\sqrt{-1}}{n}}$, $u=e^{\frac{1}{n}U}$ and $v=e^{\frac{1}{n}V}$ for every $n$. If $\delta>0$ is sufficiently small and $n$ is sufficiently large, then
$$\big| \mathrm{QDL}^q(u,v |j) \big| = O\Big( e^{\frac{n}{2\pi}\Lambda(2\delta)}\Big)$$
whenever $-\frac{\pi}{2} -\delta \leq \frac{2\pi j}{n} \leq -\frac{\pi}{2} +\delta$ or $\frac{\pi}{2} - \delta \leq \frac{2\pi j}{n} \leq \frac{\pi}{2}-\delta$.
Moreover, the constants hidden in the condition ``$n$ sufficiently large'' and in big O notation can be chosen to depend only on $U,V$ and $\delta$.
\end{proposition}
Define a function $D^q(u)$ by
\begin{align}\label{defD}
D^q(u) = \prod_{j=1}^n \mathrm{QDL}^q(u,v | j).
\end{align}
This function will show up as part of the normalization factor of the invariants. Moreover, its asymptotics is given by the following proposition.
\begin{proposition}[Proposition 18, \cite{BWYII}]\label{asymDqu}
Let $A\in \CC$ be given with $e^{A} \neq -1$. For every odd $n$, set $q=e^{\frac{2\pi \sqrt{-1}}{n}}$ and $u = qe^{-\frac{1}{n}A}$. Then
\begin{align*}
\lim_{\substack{n\to \infty\\ n \equiv 1 \text{ }\mathrm{mod}{4}}} |D^q(u)|^{\frac{1}{n}} &= 2^{\frac{-\mathrm{Im} A}{4\pi}}
\Bigg| \frac{\cosh \frac{A-\pi\sqrt{-1}}{4}}{\cosh \frac{A+\pi\sqrt{-1}}{4}} \Bigg|^{\frac{1}{4}}, \\
\lim_{\substack{n\to \infty\\ n \equiv 3 \text{ }\mathrm{mod}{4}}} |D^q(u)|^{\frac{1}{n}} &= 2^{\frac{-\mathrm{Im} A}{4\pi}}
\Bigg| \frac{\sinh \frac{A-\pi\sqrt{-1}}{4}}{\sinh \frac{A+\pi\sqrt{-1}}{4}} \Bigg|^{\frac{1}{4}} .
\end{align*}
\end{proposition}

\subsection{Continuous quantum dilogarithm function}

For every odd $n$, the \emph{small continuous quantum dilogarithm function} $\mathrm{li}_2^{\frac{2}{n}}(z)$ is defined by
$$ \mathrm{li}_2^{\frac{2}{n}}(z) = \frac{4\pi\sqrt{-1}}{n} \int_{\Omega} \frac{e^{(2z-\pi)t}}{4t\sinh(\pi t) \sinh\big(\frac{2\pi t}{n}\big)} dt,$$
where $-\frac{\pi}{n} < \mathrm{Re} z < \pi + \frac{\pi}{n}$ and $\Omega$ is a contour given by
$$ \Omega = (-\infty, -\epsilon] \cup \Big\{ \epsilon e^{\sqrt{-1}\theta} \mid \theta \in [-\pi, \pi] \Big\} \cup [\epsilon, \infty) $$
for $\epsilon > 0$. Besides, the \emph{big continuous quantum dilogarithm function} $\mathrm{Li}_2^{\frac{2}{n}}(z)$ is defined by
$$ \mathrm{Li}_2^{\frac{2}{n}}(z) = e^{\frac{n}{4\pi\sqrt{-1}} \mathrm{li}_2^{\frac{2}{n}}(z)} $$
for $ -\frac{\pi}{n} < \mathrm{Re} z < \pi + \frac{\pi}{n}$.
\begin{proposition}[Proposition 4, \cite{BWYII}] \label{QDL2}
The function $\mathrm{Li}_2^{\frac{2}{n}}$ uniquely extends to a meromorphic function on the whole complex plane, whose poles are the points $\pi+\frac{\pi}{n}+a\pi + \frac{2b\pi}{n}$ for all integers $a,b \geq 0$, and whose zeros are the points $-\frac{\pi}{n}-a\pi - \frac{2b\pi}{n}$ for all integers $a,b \geq 0$.
This extension satisfies the functional equations
$$ \mathrm{Li}^{\frac{2}{n}}_2(z+ \pi) =  \Big(1+e^{\frac{\sqrt{-1}z}{n}}\Big)^{-1} \mathrm{Li}^{\frac{2}{n}}_2(z). $$
\end{proposition}

\subsection{Classical dilogarithm function and Lobachesky function}

For $z\in \CC \setminus [1,\infty)$, the \emph{classical dilogarithm function} $\mathrm{li}_2(z)$ is given by
$$ \mathrm{li}_2(z) = -\int_0^z \frac{\log(1-t)}{t} dt.$$
The discrete quantum dilogarithm function and the continuous quantum dilogarithm function are related as follows.
\begin{proposition}[Corollary 5, \cite{BWYII}] \label{QDL1}
Let $U, V \in \CC$ satisfying $e^U = 1+e^V$. For every odd integer $n$, let $q=e^{\frac{2\pi i }{n}}, u=e^{\frac{1}{n}U}$ and $v=e^{\frac{1}{n}V}$. For every $j \in \ZZ$,
$$ \mathrm{QDL}^q(u,v |j) = e^{-\frac{j}{n}V} \frac{\mathrm{Li}_2^{\frac{2}{n}}\Big(\frac{\pi}{2}-\frac{\pi}{n}+\frac{U}{2\pi i}-\frac{2\pi j}{n}\Big)}{ \mathrm{Li}_2^{\frac{2}{n}}\Big(\frac{\pi}{2}-\frac{\pi}{n}+\frac{U}{2\pi i}\Big)}.$$
\end{proposition}
Besides, the continuous quantum dilogarithm function and the classical dilogarithm function are related as follows.
\begin{proposition}[Proposition 3, \cite{BWYII}] \label{conttoclass}
For every $z$ with $0<\mathrm{Re} z < \pi$,
$$ \mathrm{li}_2^{\frac{2}{n}}(z) = \mathrm{li}_2\Big(e^{2\sqrt{-1}z}\Big) + O\bigg(\frac{1}{n^2}\bigg),$$
where the constant in the big O notation can be chosen to be dependent on $z$ uniformly on any compact subset.
\end{proposition}

Recall that the \emph{Lobachevsky function} $\Lambda: \RR \to \RR$ is defined by
$$ \Lambda(\theta) = - \int_0^{\theta} \log|2\sin t| dt .$$
It is well-known that the Lobachevsky function and the classical dilogarithm function are related as follows.
\begin{proposition}\label{2loba} For $\theta \in [0,\pi]$,
$$\mathrm{li}_2 \Big(e^{2\sqrt{-1}\theta}\Big)
= \frac{\pi^2}{6} - \theta(\pi-\theta) + 2\sqrt{-1}\Lambda(\theta).$$
In particular, $\mathrm{Im} \Big( \mathrm{li}_2 \Big(e^{2\sqrt{-1}\theta}\Big) \Big) = 2\Lambda(\theta)$ for $\theta \in [0,\pi]$.
\end{proposition}

\section{Explicit Formula of the quantum invariants of surface diffeomorphisms}\label{EFQISD}
Let $\Sigma=\Sigma_{1,1}$ be the once-punctured torus bundle. It is known that the action of the mapping class group $\mathrm{Mod}(\Sigma)$ of $\Sigma$ is determined by the action on the first homology group $\mathrm{H}_{1}(\Sigma,\ZZ) \cong \ZZ^2$. In particular, $\mathrm{Mod}(\Sigma) \cong \mathrm{SL}(2;\ZZ)$. Moreover, every conjugacy class in $ \mathrm{SL}(2;\ZZ)$ admits a representative of the form $\varphi = \pm \varphi_1 \circ \dots \circ \varphi_k$, where each $\varphi_i$ is equal to $L=\begin{pmatrix} 1&1 \\ 0 &1\end{pmatrix}$ or  $R=\begin{pmatrix} 1&0 \\ 1 &1\end{pmatrix}$.

Let $\varphi=LR$ be the surface diffeomorphism and let $M_\varphi$ be the mapping torus given by
$$M_\varphi = \big( \Sigma_{1,1} \times [0,1] \big) / (0, x) \sim (1, \varphi(x)).$$
It is known that $M_\varphi$ is homeomorphic to the figure eight knot complement, since they share the same ideal triangulation \cite{T88, Gue}.
Let $r \in \chi_{\mathrm{PSL}(2;\CC)}(\Sigma_{1,1})$ be a generic $\mathrm{PSL}(2;\CC)$ character with periodic edge weight system (see \cite{BWYI})
$$(a_0, b_0, c_0), (a_1, b_1, c_1), (a_2, b_2, c_2) \in  (\CC^*)^3$$
such that 
\begin{alignat}{3}
a_1 &= b_0^{-1} \qquad  &&b_1= (1+b_0)^{2}a_0 \qquad  &&c_1= (1+b_0)^{-2} b_0^2 c_0 \label{period1}\\
a_2 &= c_1^{-1} \qquad  &&b_2= (1+c_1)^{2}b_1 \qquad  &&c_2= (1+c_1)^{-2} c_1^2 a_1 \label{period2}\\
a_0 &= a_2 \qquad  &&b_0 = b_2 \qquad &&c_0 = c_2.
\end{alignat}
To compute the invariant, we also need to choose logarithms $A_k, B_k, C_k, V_k \in \CC^*$ such that for $k=0,1,2$,
$$ e^{A_k}= a_k, \quad e^{B_k}= b_k, \quad e^{C_k}= c_k, \quad e^{V_k} = 1+a_k^{-1}$$
and
\begin{alignat*}{2}
&A_1= -B_0,  && A_2 = -C_1,\\
&B_1=2V_1+A_0, && B_2 =2V_2+B_1, \\
&C_1=-2V_1 +2B_0 + C_0, \qquad&& C_2 = -2V_2 +2C_1 + A_1.
\end{alignat*}
By the periodicity, there exists integers $\hat{l}_0, \hat{m}_0, \hat{n}_0 \in \ZZ$ such that
$$ A_0 = A_2 + 2\pi\hat{l}_0\sqrt{-1} , \quad B_0 = B_2 + 2\pi\hat{m}_0\sqrt{-1} , \quad C_0 = C_2 + 2\pi\hat{n}_0\sqrt{-1} $$
with $\hat{l}_0+\hat{m}_0+\hat{n}_0=0$.

\begin{proposition}[Proposition 2 \& Equation (1), \cite{BWYII}]\label{Explicitfor}
For the diffeomorphism $\varphi=LR$, let $r$  be the $\varphi$-invariant $\mathrm{PSL}(2;\CC)$-character associated to the periodic edge weight system $(a_0, b_0, c_0)$, $(a_1, b_1, c_1)$, $(a_2, b_2, c_2) \in  (\CC^*)^3$, and let $e^{h_v}=a_0b_0c_0$. For every odd integer $n$, let $\Lambda^q_{\varphi,r, p_v}$ be the intertwiner associated to the data $\varphi$, $r$, $q=e^{\frac{2\pi \sqrt{-1}}{n}}$ and $p_v = e^{\frac{1}{n}h_v} + e^{-\frac{1}{n}h_v}$. Then, up to a multiplication by a scalar with modulus 1,
\begin{align*}
 \mathrm{Trace}\text{ }(\Lambda^q_{\varphi,r, p_v})
&= \frac{1}{n\Big|D^q\Big(qe^{-\frac{A_1}{n}}\Big)\Big|^{\frac{1}{n}}\Big|D^q\Big(qe^{-\frac{A_2}{n}}\Big)\Big|^{\frac{1}{n}}} \\
&\qquad \times \sum_{i_1=1}^n \mathrm{QDL}^q\Big( qe^{-\frac{A_1}{n}}, e^{\frac{V_1}{n}} \Big| 2i_1\Big) q^{2i_1^2-\hat{l}_0 i_1} \\
&\qquad\qquad \times \sum_{i_2=1}^n \mathrm{QDL}^q\Big( qe^{-\frac{A_2}{n}}, e^{\frac{V_2}{n}} \Big| 2i_2 \Big) q^{2i_2^2-\hat{m}_0 i_2}.
\end{align*}
\end{proposition}

\subsection{Computation of $(\tilde l, \tilde m, \tilde n)$}
In the periodic edge weight system with $a_0 b_0 c_0 = e^{h_v} $, we see the phase difference between $(a_2,b_2,c_2)$ and $(a_0,b_0,c_0)$ is captured by a 3-tuple $(\hat{l},\hat{m},\hat{n})$. For our case of periodic edge weight system, with $a_0 b_0 c_0 = e^{h_v + 2\pi m_v^{(n)} \sqrt{-1}}$, we aim to calculate the phase difference by a 3-tuple $(\tilde l, \tilde m, \tilde n)$.

Let $\varphi = LR$ and let $r \in \chi_{\mathrm{PSL}(2;\CC)}(\Sigma_{1,1})$ be the $\varphi$-invariant character associated to the periodic edge weight system $(a_0, b_0, c_0)$, $(a_1, b_1, c_1)$, $(a_2, b_2, c_2) \in  (\CC^*)^3$. We fix $h_v$ such that $e^{h_v}=a_0b_0c_0$.
Let $A_k, B_k, C_k, V_k$ be defined as before, i.e. $a_k=e^{A_k}, b_k = e^{B_k}, c_k=e^{C_k}, e^{V_k} = 1+a_k^{-1}$ with $A_0+B_0+C_0 = h_v$,
$$ A_1 = -B_0, \quad B_2 = 2V_1+A_0, \quad C_1= - 2V_1 + 2B_0 + C_0,$$
$$ A_2= -C_1, \quad B_2 = 2V_2 + B_1, \quad C_2 = -2V_2 + 2C_2 + A_1.$$
Let $(\hat{l}_0, \hat{m}_0, \hat{n}_0)\in \ZZ^3$ with $\hat{l}_0+\hat{m}_0+\hat{n}_0=0$ such that $A_0=A_2 + 2\pi \hat{l}_0\sqrt{-1} $, $B_0=B_2 + 2\pi  \hat{m}_0\sqrt{-1}$ and $C_0=C_2 + 2\pi  \hat{n}_0\sqrt{-1}$.

Now, we modify the puncture weight in such a way that it depends on the parameter $n$. Let $m_{v}^{(n)}\in \{-\frac{n-1}{2}, - \frac{n-3}{2},\dots, \frac{n-1}{2}\}$ be a sequence of integers such that $\lim_{n\to \infty} \frac{4\pi m_{v}^{(n)}}{n}$ exists. We further assume that
$$ \theta_v = \Bigg| \lim_{n\to \infty} \frac{4\pi m_{v}^{(n)}}{n}\Bigg|  \in [0, 2\pi).$$
To simplify the notation, we write $m_v = m_v^{(n)}$. We are going to let $\tilde{A}_k, \tilde{B}_k, \tilde{C}_k$ be the new choices of logarithms. More precisely, let $\tilde{A}_0=A_0 + 2\pi m_v \sqrt{-1}, \tilde{B}_0=B_0, \tilde{C}_0=C_0$ with $\tilde{A}_0+\tilde{B}_0+\tilde{C}_0= h_v + 2\pi m_v\sqrt{-1} $. By direct computation, we have
\begin{align*}
\tilde{A}_1&=-\tilde{B}_0 = -B_0 = A_1  \\
\tilde{B}_1 &= 2V_1 + \tilde{A}_0 = 2V_1 + A_0 + 2\pi m_v \sqrt{-1} = B_1 + 2\pi m_v \sqrt{-1},\\
\tilde{C}_1 &= -2V_1 + 2\tilde{B}_0 + \tilde{C}_0 = -2V_1 + 2B_0 + C_0  = C_1
\end{align*}
and
\begin{align*}
\tilde{A}_2&=-\tilde{C}_1 = -C_1  = A_2  \\
\tilde{B}_2&= 2V_2 + \tilde{B}_1 = 2V_2 + B_1 + 2\pi m_v \sqrt{-1} = B_2+ 2\pi m_v \sqrt{-1},\\
\tilde{C}_2 &= -2V_2 + 2\tilde{C}_1 + \tilde{A}_1 = -2V_2 + 2C_1 + A_1 = C_2
\end{align*}
Since
\begin{align*}
\tilde{A}_0&=A_0 + 2\pi  m_v  \sqrt{-1} = A_2 + 2\pi \sqrt{-1} \hat{l}_0 + 2\pi m_v  \sqrt{-1}= \tilde{A}_2 + 2\pi (\hat{l}_0 + m_v )\sqrt{-1} , \\
\tilde{B}_0&= B_0 = B_2 + 2\pi \hat{m}_0 \sqrt{-1} = \tilde{B}_2 +2\pi ( \hat{m}_0 - m_v)\sqrt{-1},\\
\tilde{C}_0&= C_0 = C_2 + 2\pi \hat{n}_0  \sqrt{-1} = \tilde{C}_2 + 2\pi  \hat{n}_0 \sqrt{-1},
\end{align*}
we have $(\tilde{l}, \tilde{m}, \tilde{n}) = (\hat{l}_0 + m_v, \hat{m}_0 - m_v, \hat{n}_0 ) \in \ZZ^3$ with $\tilde{l}+\tilde{m}+\tilde{n} = \hat{l}_0 +  \hat{m}_0 + \hat{n}_0 = 0$.

\subsection{Explicit formula of $\mathrm{Trace}\text{ } \Lambda_{\varphi, r, p_v}^q $}
By Proposition \ref{Explicitfor}, for the puncture weight $p_v = -e^{(h+2\pi m_v \sqrt{-1} )/n} - e^{-(h+2 m_v\pi \sqrt{-1} )/n}$, we have
\begin{align}\label{Expfor}
\mathrm{Trace} \Lambda_{\varphi, r, p_v}^q = \frac{1}{n\Big|D^q\Big(qe^{-\frac{A_1}{n}}\Big)\Big|^{\frac{1}{n}}\Big|D^q\Big(qe^{-\frac{A_2}{n}}\Big)\Big|^{\frac{1}{n}}} \times \Sigma_1 \times \Sigma_2,
\end{align}
where
\begin{align*}
\Sigma_1 = \sum_{i_1=1}^n \mathrm{QDL}^q\Big( qe^{-\frac{A_1}{n}}, e^{\frac{V_1}{n}} \Big| 2i_1\Big) q^{2i_1^2-\hat{l}_0 i_1 - m_v i_1}
\end{align*}
and
\begin{align*}
\Sigma_2 = \sum_{i_2=1}^n \mathrm{QDL}^q\Big( qe^{-\frac{A_2}{n}}, e^{\frac{V_2}{n}} \Big| 2i_2 \Big) q^{2i_2^2-\hat{m}_0 i_2+ m_v i_2}.
\end{align*}
Define regions $R_{1,1}, R_{1,2}, R_{2,1}, R_{2,2}$ by
\begin{align*}
R_{1,1} &= \Big\{ j \mid j \in \{1,2,\dots, n\},  0 \leq \mathrm{Re}\Big(\frac{\pi}{2} + \frac{A_1}{2ni} - \frac{2\pi j}{n}\Big)< \pi \Big\}\\
R_{1,2} &= \Big\{ j  \mid j \in \{1,2,\dots, n\}, -\pi \leq  \mathrm{Re}\Big(\frac{\pi}{2} + \frac{A_1}{2ni} - \frac{2\pi j}{n}\Big)< 0 \Big\}\\
R_{2,1} &= \Big\{ j \mid j \in \{1,2,\dots, n\},  0 \leq \mathrm{Re}\Big(\frac{\pi}{2} + \frac{A_2}{2ni} - \frac{2\pi j}{n}\Big)< \pi \Big\}\\
R_{2,2} &= \Big\{ j  \mid j \in \{1,2,\dots, n\}, -\pi \leq  \mathrm{Re}\Big(\frac{\pi}{2} + \frac{A_2}{2ni} - \frac{2\pi j}{n}\Big)< 0 \Big\}
\end{align*}
respectively. Let
\begin{align*}
\Sigma_{1,1}
&= \sum_{i_1 \in R_{1,1}} (-1)^{(1-\hat{l}_0 - m_v)i_1}e^{\frac{-i_1 V_1}{n}} \frac{\mathrm{Li}_2^{\frac{2}{n}} \Big(\frac{\pi}{2} + \frac{A_1}{2n\sqrt{-1}} - \frac{2\pi i_1}{n}\Big)}{\mathrm{Li}_2^{\frac{2}{n}}\Big(\frac{\pi}{2} + \frac{A_1}{2n\sqrt{-1}}  \Big)} e^{\frac{\pi \sqrt{-1}}{n} (i_1^2 - \hat{l}_0 i_1 - m_v i_1)},\\
\Sigma_{1,2}
&= \Big(1 +  (\sqrt{-1})^n e^{\frac{A_1}{2}} \Big) \\
&\quad\times  \sum_{i_1 \in R_{1,2}} (-1)^{(1-\hat{l}_0 - m_v)i_1}e^{\frac{-i_1 V_1}{n}} \frac{\mathrm{Li}_2^{\frac{2}{n}} \Big(\frac{3\pi}{2} + \frac{A_1}{2n\sqrt{-1}} - \frac{2\pi i_1}{n} \Big)}{\mathrm{Li}_2^{\frac{2}{n}}\Big(\frac{\pi}{2} + \frac{A_1}{2n\sqrt{-1}}  \Big)} e^{\frac{\pi \sqrt{-1}}{n} (i_1^2 - \hat{l}_0 i_1 - m_v i_1)},\\
\Sigma_{2,1}
&= \sum_{i_2 \in R_{2,1}} (-1)^{(1-\hat{m}_0-m_v)i_2} e^{\frac{-i_2 V_2}{n}} \frac{\mathrm{Li}_2^{\frac{2}{n}} \Big(\frac{\pi}{2} + \frac{A_2}{2n\sqrt{-1}} - \frac{2\pi i_2}{n}\Big)}{\mathrm{Li}_2^{\frac{2}{n}}\Big(\frac{\pi}{2} + \frac{A_2}{2n\sqrt{-1}}  \Big)} e^{\frac{\pi \sqrt{-1}}{n} (i_2^2 - \hat{m}_0 i_2 + m_vi_2)},\\
\Sigma_{2,2}
&= \Big(1+  (\sqrt{-1})^n e^{\frac{A_2}{2}}\Big) \\
&\quad\times \sum_{i_2 \in R_{2,2}} (-1)^{(1-\hat{m}_0-m_v)i_2} e^{\frac{-i_2 V_2}{n}} \frac{\mathrm{Li}_2^{\frac{2}{n}} \Big(\frac{3\pi}{2} + \frac{A_2}{2n\sqrt{-1}} - \frac{2\pi i_2}{n}\Big)}{\mathrm{Li}_2^{\frac{2}{n}}\Big(\frac{\pi}{2} + \frac{A_2}{2n\sqrt{-1}}  \Big)} e^{\frac{\pi \sqrt{-1}}{n} (i_2^2 - \hat{m}_0 i_2 + m_v i_2)}.
\end{align*}
\begin{proposition}\label{Sigmasum0} We have
\begin{align*}
\Sigma_{1} =\Sigma_{1,1} + \Sigma_{1,2} \quad\text{and}\quad
\Sigma_{2} = \Sigma_{2,1} + \Sigma_{2,2}.
\end{align*}
\end{proposition}
\begin{proof} Let $\omega = -e^{\frac{\pi \sqrt{-1}}{n}}$. Since the function $i_1 \mapsto \mathrm{QDL}^q\Big( qe^{\frac{-A_1}{n}}, e^{\frac{V_1}{n}} \Big| i_1\Big) \omega^{i_1^2 - \hat{l}_0 i_1 - m_v i_1}$ is $n$-periodic, we have
\begin{align*}
\Sigma_1
&= \sum_{i_1=1}^n \mathrm{QDL}^q\Big( qe^{-\frac{A_1}{n}}, e^{\frac{V_1}{n}} \Big| 2i_1\Big) q^{2i_1^2-\hat{l}_0 i_1 - m_v i_1}\\
&= \sum_{i_1=1}^n \mathrm{QDL}^q\Big( qe^{-\frac{A_1}{n}}, e^{\frac{V_1}{n}} \Big| 2i_1\Big) \omega^{(2i_1)^2-\hat{l}_0 (2i_1) - m_v (2i_1)}\\
&= \sum_{i_1=1}^n \mathrm{QDL}^q\Big( qe^{-\frac{A_1}{n}}, e^{\frac{V_1}{n}} \Big| i_1\Big) \omega^{i_1^2-\hat{l}_0 i_1 - m_v i_1}
\end{align*}
The formula for $\Sigma_1$ follows from Propositions \ref{QDL2} and \ref{QDL1}. The formula for $\Sigma_2$ follows from a similar computation.
\end{proof}

For $\delta > 0$ small and $s,t\in \{1,2\}$, let
\begin{align*}
R_{s,t}^{\delta} &= \Big\{ \alpha_s \in \CC \mid (1-t)\pi+  \delta \leq \mathrm{Re}\Big(\frac{\pi}{2}-  \alpha_s\Big)< (2-t)\pi-\delta \Big\}.
\end{align*}

Let $C_{s,t}$ be the contour defined by $C_{s,t} = R_{s,t}^{\delta} \cap \RR.$ Let $b_{s,t}:\RR \to [0,1]$ be a bump function with $b_{s,t}(x) = 1$ for $x\in C_{s,t}= R^\delta_{s,t} \cap \RR$, $b_{s,t}(x) =0$ for $x \in \RR \setminus \RR^0_{s,t}$ and $b_{s,t}(x) \in (0,1)$ otherwise. Let $p_1, p_2 \in \{0, 1\}$ such that  $p_1 = 1-\hat{l}_0-m_v  \pmod{2}$ and $p_2 = 1-\hat{m}_0 -m_v \pmod{2}$.

\begin{proposition}\label{Sigmasum} We have
\begin{align*}
\Sigma_{1,1}
&=\frac{n}{2\pi}\sum_{k_1 \in \ZZ} F_{1,1}(k_1)
+ \mathrm{Error} ,\\
\Sigma_{1,2}
&= \Big(1 +  (\sqrt{-1})^n e^{\frac{A_1}{2}}\Big) \times \frac{n}{2\pi}\sum_{k_1 \in \ZZ} F_{1,2}(k_1)
+ \mathrm{Error} ,\\
\Sigma_{2,1}
&=\frac{n}{2\pi}\sum_{k_2 \in \ZZ} F_{2,1}(k_2)
+ \mathrm{Error} ,\\
\Sigma_{2,2}
&= \Big(1 +  (\sqrt{-1})^n e^{\frac{A_2}{2}}\Big) \times \frac{n}{2\pi}\sum_{k_1 \in \ZZ} F_{2,1}(k_2)
+ \mathrm{Error} ,
\end{align*}
where for $s,t\in \{1,2\}$ and $k\in\ZZ$,
\begin{align*}
F_{s,t} (k) = \int_{C_{s,t}} g_s(\alpha_s) e^{\frac{n}{4\pi\sqrt{-1}}  \Big(f_s(\alpha_s; \eta_v)
+ 2p_s\pi  \alpha  -4k\pi  \alpha_s + O(\frac{1}{n})\Big)} d\alpha_s
\end{align*}
with $\eta_v = \frac{2\pi m_v}{n}$,
\begin{align*}
f_{1} (\alpha_1; \eta_v) &=  \mathrm{li}_2 \Big(-e^{- 2\alpha_1\sqrt{-1}}\Big) - \alpha_1^2 + \eta_v \alpha_1 + \frac{\pi^2}{12} ,  \\
g_{1} (\alpha_1) &= e^{-\big(\frac{\hat{l}_0 \sqrt{-1}}{2} + \frac{V_1}{2\pi}\big) \alpha_1} \Big(1+e^{-2\alpha_1\sqrt{-1}} \Big)^{\frac{A_1}{4\pi \sqrt{-1}}},\\
f_{2} (\alpha_2; \eta_v) &=  \mathrm{li}_2 \Big(-e^{- 2\alpha_2\sqrt{-1}}\Big) - \alpha_2^2 - \eta_v \alpha_2 + \frac{\pi^2}{12} ,  \\
g_{2} (\alpha_2) &= e^{-\big(\frac{\hat{m}_0 \sqrt{-1}}{2} + \frac{V_2}{2\pi}\big) \alpha_2 } \Big(1+e^{-2\alpha_2\sqrt{-1}} \Big)^{\frac{A_2}{4\pi \sqrt{-1}}}.
\end{align*}
Moreover, all the error terms satisfy
$ |\mathrm{Error} | = O\Big(e^{\frac{n}{2\pi} \Lambda(2\delta)}\Big).$
\end{proposition}

\begin{proof}
Note that for $i_1=0,1,\dots, n-1$,
\begin{align*}
e^{- \frac{i_1V_1}{n} + \frac{\pi \sqrt{-1}}{n} ( i_1^2 - \hat{l}_0 i_1 - m_v i_1)} = e^{-\big(\frac{\hat{l}_0\sqrt{-1}}{2} + \frac{V_1}{2\pi}\big) \big(\frac{2\pi i_1}{n}\big) -  \frac{n}{4\pi \sqrt{-1}} \Big[\Big(\frac{2\pi i_1}{n}\Big)^2 - \Big(\frac{2\pi m_v}{n}\Big)\Big(\frac{2\pi i_1}{n}\Big)\Big]}.
\end{align*}

Besides,
$$(-1)^{(1-\hat{l}_0 -m_v)i_1}
= e^{\frac{n}{4\pi\sqrt{-1}} (2p_1 \pi ) \Big(\frac{2\pi i_1}{n}\Big) }.
$$

By the Poisson Summation Formula (see e.g. \cite[Theorem 3.1]{SS}), we can write
\begin{align*}
\Sigma_{1,1}
= &\sum_{i_1\in \ZZ} \left[b_{1,1}\Bigg(\frac{2\pi i_1}{n}\Bigg)e^{-\big(\frac{\hat{l}_0\sqrt{-1}}{2} + \frac{V_1}{2\pi}\big) \big(\frac{2\pi i_1}{n}\big) -  \frac{n}{4\pi \sqrt{-1}} \Big[\Big(\frac{2\pi i_1}{n}\Big)^2 - \Big(\frac{2\pi m_v}{n}\Big)\Big(\frac{2\pi i_1}{n}\Big) - (2p_1 \pi ) \Big(\frac{2\pi i_1}{n}\Big)\Big]} \right.\\
&\qquad\qquad \left. \times \frac{\mathrm{Li}_2^{\frac{2}{n}}\Big(\frac{\pi}{2} + \frac{A_1}{2n\sqrt{-1}}  - \frac{2\pi i_1}{n}\Big)}{\mathrm{Li}_2^{\frac{2}{n}}\Big(\frac{\pi}{2} + \frac{A_1}{2n\sqrt{-1}} \Big)} \right]\\
& +\text{Error} \\
=&\sum_{k_1 \in \ZZ} \int_{\RR}  \left[b_{1,1}\Bigg(\frac{2\pi i_1}{n}\Bigg)e^{-\big(\frac{\hat{l}_0\sqrt{-1}}{2} + \frac{V_1}{2\pi}\big) \big(\frac{2\pi i_1}{n}\big) -  \frac{n}{4\pi \sqrt{-1}} \Big[\Big(\frac{2\pi i_1}{n}\Big)^2 - \Big(\frac{2\pi m_v}{n}\Big)\Big(\frac{2\pi i_1}{n}\Big) - (2p_1 \pi ) \Big(\frac{2\pi i_1}{n}\Big)\Big]} \right.\\
&\qquad\qquad\qquad \left.\times \frac{\mathrm{Li}_2^{\frac{2}{n}}\Big(\frac{\pi}{2} + \frac{A_1}{2n\sqrt{-1}}  - \frac{2\pi i_1}{n}\Big)}{\mathrm{Li}_2^{\frac{2}{n}}\Big(\frac{\pi}{2} + \frac{A_1}{2n\sqrt{-1}} \Big)} e^{2\pi k_1  i_1 \sqrt{-1}} \right]di_1 \\
& +\text{Error} .
\end{align*}
By using the substitution $\alpha_1 = \frac{2\pi i_1}{n}$, we have
\begin{align*}
\Sigma_{1,1}
=&\frac{n}{2\pi}\sum_{k_1 \in \ZZ} \int_{C_{1,1}} \Bigg(e^{-\big(\frac{\hat{l}_0 \sqrt{-1}}{2} + \frac{V_1}{2\pi}\big) \alpha_1 - \frac{n}{4\pi \sqrt{-1}} \Big[\alpha_1^2 - \big(\frac{2\pi m_v}{n}\big)\alpha_1 - 2p_1\pi \alpha_1 \Big]} \\
&\qquad \qquad\qquad\qquad \times \frac{\mathrm{Li}_2^{\frac{2}{n}}\Big(\frac{\pi}{2} + \frac{A_1}{2n\sqrt{-1}}  - \alpha_1\Big)}{\mathrm{Li}_2^{\frac{2}{n}}\Big(\frac{\pi}{2} + \frac{A_1}{2n\sqrt{-1}} \Big)} e^{k_1 \alpha_1 \sqrt{-1}} \Bigg) d\alpha_1 \\
&+ \text{Error}.
\end{align*}

For all $\alpha_1 \in R_{1,1}^\delta$, by Proposition \ref{conttoclass},
\begin{align*}
\mathrm{Li}_2^{\frac{2}{n}}\Bigg(\frac{\pi}{2} + \frac{A_1}{2n\sqrt{-1}} - \alpha_1\Bigg)
&= \exp\Bigg(\frac{n}{4\pi \sqrt{-1}} \mathrm{li}_2^{\frac{2}{n}} \Bigg(\frac{\pi}{2}  + \frac{A_1}{2n\sqrt{-1}} - \alpha_1\Bigg) + O\Big(\frac{1}{n}\Big) \Bigg)\\
&= \exp\Big(\frac{n}{4\pi \sqrt{-1}} \mathrm{li}_2 \Bigg(-e^{ \frac{A_1}{n}  - 2\alpha_1 \sqrt{-1}}\Bigg) + O\Big(\frac{1}{n}\Big)\Big)\\
&=  \Big(1+e^{-2\alpha_1\sqrt{-1}} \Big)^{\frac{A_1}{4\pi \sqrt{-1}}} \exp\Bigg(\frac{n}{4\pi \sqrt{-1}} \mathrm{li}_2 \Big(-e^{ - 2\alpha_1\sqrt{-1}}\Big) + O\Big(\frac{1}{n}\Big)\Bigg).
\end{align*}
In particular, by Proposition \ref{2loba} and the fact that $\Lambda(\frac{\pi}{2}) =0$,
\begin{align*}
\mathrm{Li}_2^{\frac{2}{n}}\Bigg(\frac{\pi}{2} + \frac{A_1}{2n\sqrt{-1}} \Bigg)
&=   \exp\Bigg(\frac{n}{4\pi \sqrt{-1}} \mathrm{li}_2 (-1)+ O\Big(\frac{1}{n}\Big)\Bigg)
= \exp\Bigg(\frac{n}{4\pi \sqrt{-1}}  \Bigg(-\frac{\pi^2}{12}\Bigg) + O\Big(\frac{1}{n}\Big)\Bigg).
\end{align*}

As a result, we have
\begin{align*}
\Sigma_{1,1}
=\frac{n}{2\pi}&\sum_{k_1 \in \ZZ} F_{1,1}(k) + \text{Error},
\end{align*}
where $\eta_v =\frac{2\pi m_v}{n}$,
\begin{align*}
F_{1,1}(k) &=  \int_{C_{1,1}} g_1(\alpha_1) e^{\frac{n}{4\pi\sqrt{-1}} \Big(f_1(\alpha_1; \eta_v) + 2p_1 \pi \alpha_1 -4k\pi \alpha_1 + O(\frac{1}{n})\Big)} d\alpha_1 ,\\
f_{1} (\alpha_1; \eta_v) &=  \mathrm{li}_2 \Big(-e^{- 2\alpha_1\sqrt{-1}}\Big) - \alpha_1^2 + \eta_v \alpha_1 + \frac{\pi^2}{12} ,\\
g_{1} (\alpha_1) &= e^{-\big(\frac{\hat{l}_0 \sqrt{-1}}{2} + \frac{V}{2\pi}\big) \alpha_1 } \Big(1+e^{-2\alpha_1\sqrt{-1}} \Big)^{\frac{A_1}{4\pi \sqrt{-1}}} .
\end{align*}

Finally, note that the error consists of terms satisfying the condition in Proposition \ref{Estbridge}. The estimation follows from Proposition \ref{Estbridge}. The computation for $\Sigma_{1,2}, \Sigma_{2,1}$ and $\Sigma_{2,2}$ are similar.
\end{proof}

\section{Geometry of the potential functions}\label{GPF}
\subsection{Geometric triangulation for the hyperbolic cone metric on $\Sigma_\eta$}
Consider the ideal triangulation $\mathcal{T}$ of the once-punctured torus bundle given in \cite{Gue}. Let $z_1,z_2$ be the shape parameters of the ideal tetrahedra and let $\mu_1, \mu_2$ be the curves as shown in Figure \ref{boundary}.
\begin{figure}[t]
\includegraphics[width=15cm]{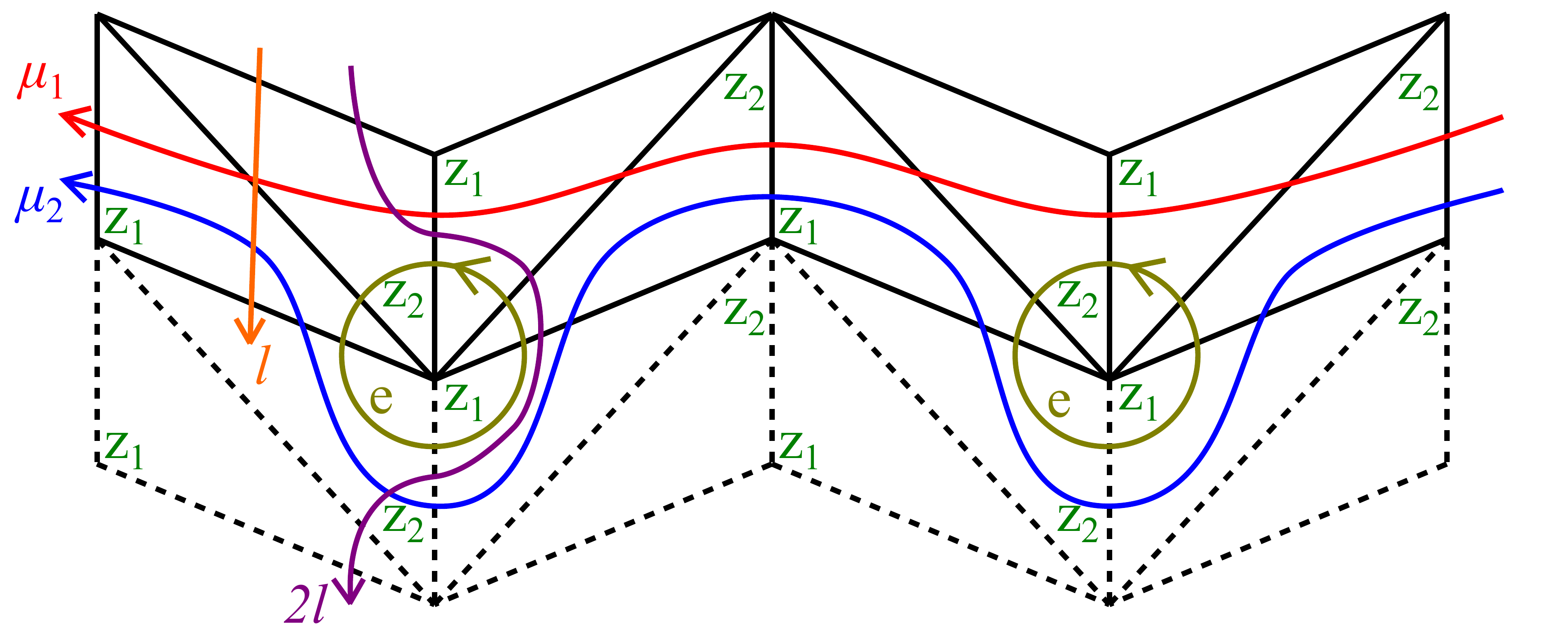}
\caption{Triangulation of the boundary torus. The part drawn by solid lines and the one drawn by dotted lines are copies of the fundamental domain. The red curve $\mu_1$ and the blue curve $\mu_2$ are homotopic curves representing the meridian $\mu_v$. The orange curve represents a curve $l$ that intersects the meridian once. The purple curve represents $2l$.}\label{boundary}
\centering
\end{figure}
Note that the logarithmic holonomies of $\mu_1$ and $\mu_2$ are given by
\begin{align}
 \mathrm{H}(\mu_1) &= 2\big(\log (z_1') - \log(z_1'') \big), \label{Hm1}\\
 \mathrm{H}(\mu_2) &= -2 \big(\log(z_2') - \log(z_2'')  \big) \label{Hm2},
\end{align}
where $z_i'=\frac{1}{1-z_i}$ and $z_i''=1-\frac{1}{z_i}$ for $i=1,2$.
Moreover, the holonomy of the curve $e$ is given by
\begin{align}
\mathrm{H}(e) =  \log (z_1) + 2\log(z_1') + \log(z_2) + 2\log(z_2''). \label{He}
\end{align}
Especially, we have
\begin{align}\label{diffholedge}
 \mathrm{H}(\mu_1)  -  \mathrm{H}(\mu_2)  = 2\big(\mathrm{H}(e) - 2\pi\sqrt{-1}\big).
\end{align}
Note that when $z_1=z_2= e^{\frac{\pi\sqrt{-1}}{3}}$, we have $\mathrm{H}(\mu_1)=\mathrm{H}(\mu_2) =0$ and $\mathrm{H}(e)=2\pi\sqrt{-1}$. The shape parameters $(z_1,z_2)=\big(e^{\frac{\pi\sqrt{-1}}{3}},e^{\frac{\pi\sqrt{-1}}{3}}\big)$ correspond to the complete hyperbolic structure on the once-punctured torus bundle.

For any $\eta_v\in(-\pi,\pi)$, consider the system of equations
\begin{align}\label{coneeq}
\mathrm{H}(\mu_1) = \mathrm{H}(\mu_2) = 2\eta_v \sqrt{-1}.
\end{align}
Note that when $\eta_v=0$, $z_1=z_2 = e^{\frac{\pi\sqrt{-1}}{3}}$ solve (\ref{coneeq}). Let $B=e^{\eta_v\sqrt{-1}}$. To solve (\ref{coneeq}) for $\eta_v\neq 0$, by taking exponential on both side of the equations, we have
\begin{align}
 z_1^2 - (2-B)z_1 + 1  &=0,   \label{coneqexp}\\
z_2^2 - (2-B^{-1})z_2 + 1 &=0 . \label{coneqexp2}
\end{align}
Consider the following solutions of the quadratic equations given by
$$ z_1(\eta_v)= \frac{(2-B) + \sqrt{(2-B)^2 - 4}}{2}  \quad \text{and} \quad z_2(\eta_v) = \frac{(2-B^{-1}) + \sqrt{(2-B^{-1})^2 - 4}}{2}$$
When $\eta_v=0$, $z_1(0)=z_2(0) = e^{\frac{\pi\sqrt{-1}}{3}}$ solve equations (\ref{coneeq}).

\begin{lemma}\label{geotricone}
For $\eta_v \in (-\pi, \pi)$, the ideal triangulation is geometric, i.e. $\mathrm{Im}(z_1(\eta_v)), \mathrm{Im}(z_2(\eta_v)) > 0$. In particular, for any cone angle $\theta = 2|\eta_v| \in [0,2\pi)$, $\Sigma$ admits a hyperbolic cone metric structure along the meridian with cone angle $\theta$.
\end{lemma}
\begin{proof}
Assume that $\mathrm{Im}(z_1(\eta_v))=0$ or $\mathrm{Im}(z_2(\eta_v))=0$ for some $\eta_v$. By considering the imaginary part of both sides of (\ref{coneqexp}), we have $\mathrm{Im}(B) = 0$. This can happen only when $\eta_v = k \pi $ for some $k\in \ZZ$. Note that when $\eta_v=0$, $\mathrm{Im}(z_1(0))= \mathrm{Im}(z_2(0)) = \frac{\sqrt{3}}{2} > 0$. Thus, if $\mathrm{Im}(z_1(\eta_v))=0$ or $\mathrm{Im}(z_2(\eta_v))=0$ for some $\eta_v$, we have  $\eta_v = k \pi $ for some $k\in \ZZ \setminus \{0\}$. By continuity, $\mathrm{Im}(z_1(\eta_v)),\mathrm{Im}(z_2(\eta_v))>0$ for $\eta_v \in (-\pi,\pi)$. In particular, since $\mathrm{Im}(z_1(\eta_v))$ and $\mathrm{Im}(z_2(\eta_v))$ are positive for $\eta_v\in(-\pi,\pi)$, $z_1(\eta_v)$ and $z_2(\eta_v)$ do not cross the branch cut of the logarithm for $\eta_v\in(-\pi,\pi)$. By (\ref{diffholedge}) and (\ref{coneeq}), since $(z_1(\eta_v),z_2(\eta_v))$ satifies (i) $\mathrm{Im}(z_1(\eta_v)),\mathrm{Im}(z_2(\eta_v))>0$, (ii) the edge equation $\mathrm{H}(e)=2\pi\sqrt{-1}$ and (iii) the cone angle equation, $\Sigma$ admits a hyperbolic cone metric structure along the meridian with cone angle $\theta_v = 2|\eta_v| \in [0,2\pi)$.
\end{proof}

\subsection{Critical point equations of the potential functions}
Let $\eta_v \in(-\pi,\pi)$. Recall that
\begin{align*}
f_{1}(\alpha_1; \eta_v) &=  \mathrm{li}_2 \Big(-e^{- 2\alpha_1 \sqrt{-1}}\Big) - \alpha_1^2 + \eta_v \alpha_1 + \frac{\pi^2}{12},  \\
f_{2} (\alpha_2; \eta_v) &=  \mathrm{li}_2 \Big(-e^{- 2\alpha_2 \sqrt{-1}}\Big) - \alpha_2^2 - \eta _v \alpha_2+ \frac{\pi^2}{12}.
\end{align*}
Note that the critical point equations of $f_{1}(\alpha_1; \eta_v)$ and $f_{2}(\alpha_2; \eta_v)$ with respect to $\alpha_1,\alpha_2$ are given by
\begin{align}
2\sqrt{-1}\log\Big(1+e^{-2\alpha_1 \sqrt{-1}}\Big) = 2\alpha_1 - \eta_v,  \label{ce1}\\
2\sqrt{-1}\log\Big(1+e^{-2\alpha_2 \sqrt{-1}}\Big) = 2\alpha_2 + \eta_v. \label{ce2}
\end{align}
After taking exponential on both sides, Equations (\ref{ce1}) and (\ref{ce2}) can be written as
\begin{align}
\Big(1- \Big(-e^{-2\alpha_1 \sqrt{-1}}\Big)\Big)^2 &= \Big(-e^{-2\alpha_1\sqrt{-1}}\Big) \Big(-e^{\eta_v\sqrt{-1} }\Big) \label{ceexp1}\\
\Big(1- \Big(-e^{-2\alpha_1 \sqrt{-1}}\Big)\Big)^2 &= \Big(-e^{-2\alpha_1\sqrt{-1}}\Big) \Big(-e^{-\eta_v\sqrt{-1} }\Big) \label{ceexp2}
\end{align}
Write $z_1 = -e^{-2\alpha_1 \sqrt{-1}}$, $z_2 = -e^{-2\alpha_2 \sqrt{-1}}$ and $B=e^{\eta_v \sqrt{-1}}$. Note that Equations (\ref{ceexp1}) and (\ref{ceexp2}) are equivalent to the Equations (\ref{coneqexp}) and (\ref{coneqexp2}). Define $\alpha_1(\eta_v), \alpha_2(\eta_v)$ by
$$ z_1(\eta_v) = -e^{-2\alpha_1(\eta_v)\sqrt{-1}} \quad \text \quad z_2(\eta_v) = -e^{-2\alpha_2(\eta_v)\sqrt{-1}} $$
with $\alpha_1(0) = \alpha_2(0) = \frac{\pi}{3}$. By a direct computation, $(\alpha_1(0),\alpha_2(0))$ solve equations (\ref{ce1}) and (\ref{ce2}) for $\eta_v=0$. By continuity and Lemma \ref{geotricone}, $(\alpha_1(\eta_v), \alpha_2(\eta_v))$ solve (\ref{ce1}) and (\ref{ce2}) for $\eta_v\in(-\pi, \pi)$. Altogether, we have
\begin{lemma}\label{solveeq}
$z_1(\eta_v)$ and $z_2(\eta_v)$ solve Equations (\ref{coneeq}).
\end{lemma}

\subsection{Critical values of the potential functions}
The following lemma relates the critical values of $f_1$ and $f_2$ with the hyperbolic volume of mapping torus.
\begin{lemma}\label{imgivevol} Write $\alpha_1 = x_1 + \sqrt{-1} y_1$ and $\alpha_2 = x_2 + \sqrt{-1} y_2$.
The functions $f_{1}(\alpha_1; \eta_v)$ and $f_{2}(\alpha_2; \eta_v)$ satisfy
\begin{align*}
\mathrm{Im} f_{1}(\alpha_1; \eta_v) &= D\Big(-e^{-2\alpha_1\sqrt{-1}}\Big) + y_1 \frac{\partial}{\partial y_1} \mathrm{Im} f_{1}(\alpha_1; \eta_v) ,\\
\mathrm{Im} f_{2}(\alpha_2; \eta_v) &=D\Big(-e^{-2\alpha_2\sqrt{-1}}\Big) + y_1 \frac{\partial}{\partial y_2} \mathrm{Im} f_{2}(\alpha_2; \eta_v),
\end{align*}
where $D(z)$ is the Bloch-Wigner dilogarithm function defined by
$$ D(z) = \mathrm{Im} ( \mathrm{li}_2(z)) + \mathrm{Arg}(1-z)\ln|z|.$$
In particular, we have
\begin{align*}
\mathrm{Im} \Big(f_{1}(\alpha_1(\eta_v); \eta_v) + f_{2}(\alpha_2(\eta_v); \eta_v) \Big) = \mathrm{Vol}(M_{\varphi, \theta_v}),
\end{align*}
where $\theta_v = 2|\eta_v|$.
\end{lemma}
\begin{proof}
Recall that $f_{1} (\alpha_1; \eta_v) =  \mathrm{li}_2 \Big(-e^{- 2\alpha_1 \sqrt{-1}}\Big) - \alpha_1^2 + \eta_v\alpha_1$. Note that
\begin{align*}
\frac{\partial}{\partial \alpha_1}  \mathrm{li}_2 \Big( - e^{- 2\alpha_1 \sqrt{-1}}\Big)
= 2\sqrt{-1} \log\Big(1- \Big( - e^{- 2\alpha_1 \sqrt{-1}}\Big)\Big).
\end{align*}
By the Cauchy-Riemann equation, we have
\begin{align*}
\frac{\partial}{\partial y_1}\Big(  \mathrm{Im} \Big( \mathrm{li}_2 \Big( - e^{- 2\alpha_1 \sqrt{-1}}\Big) \Big)
= \mathrm{Re} \Bigg(\frac{\partial}{\partial \alpha_1}  \mathrm{li}_2 \Big( - e^{- 2\alpha_1 \sqrt{-1}}\Big)\Bigg)
= -2 \mathrm{Arg}\Big(1- \Big( - e^{- 2\alpha_1 \sqrt{-1}}\Big)\Big).
\end{align*}
Hence,
\begin{align*}
\mathrm{Im} \Big( \mathrm{li}_2 \Big( - e^{- 2\alpha_1 \sqrt{-1}}\Big) \Big)
& = D\Big(-e^{-2\alpha_1\sqrt{-1}}\Big) - \mathrm{Arg}\Big(1- \Big( - e^{- 2\alpha_1 \sqrt{-1}}\Big)\Big) \ln \Big| - e^{- 2\alpha_1 \sqrt{-1}} \Big| \\
&= D\Big(-e^{-2\alpha_1\sqrt{-1}}\Big) - 2y_1 \mathrm{Arg}\Big(1- \Big( - e^{- 2\alpha_1 \sqrt{-1}}\Big)\Big)\\
&= D\Big(-e^{-2\alpha_1\sqrt{-1}}\Big) + y_1 \frac{\partial}{\partial y_1}\Big(  \mathrm{Im} ( \mathrm{li}_2 ) \Big)\Big( - e^{- 2\alpha_1 \sqrt{-1}}\Big).
\end{align*}
Besides,
$$ \mathrm{Im}\Bigg(\alpha_1^2 + \eta_v\alpha_1 + \frac{\pi^2}{12}\Bigg) = 2x_1y_1 + \eta_v y_1 = y_1 \frac{\partial}{\partial y_1} \mathrm{Im}\Bigg(\alpha_1^2 + \eta_v\alpha_1+ \frac{\pi^2}{12}\Bigg).
$$
The first two equations follow from a direct computation. In particular, at the critical point, we have
\begin{align*}
\mathrm{Im} \Big(f_{1}(\alpha_1(\eta_v); \eta_v) + f_{2}(\alpha_2(\eta_v); \eta_v) \Big)
&= D\Big(-e^{-2\alpha_1(\eta_v)\sqrt{-1}}\Big) + D\Big(-e^{-2\alpha_2(\eta_v)\sqrt{-1}}\Big) \\
&= \mathrm{Vol}(M_{\varphi, \theta_v}),
\end{align*}
where the last equality follows from Lemma \ref{solveeq}.
\end{proof}

\subsection{Hessians of the potential functions}
Let $\rho_{\theta_v} : \pi_1(M_\varphi) \to \mathrm{PSL}(2;\CC)$ be the holonomy representation of the cone structure on $M_\varphi$ with cone angle $\theta_v=2|\eta_v|$.  Let $\mathrm{Tor}(M_\varphi, \mu_v,  \rho_{\theta_v^{(n)}})$ be the adjoint twisted Reidemeister torsion of $M_\varphi$ with respect to the meridian $\mu_v$ and the holonomy representation $\rho_{\theta_v}$. The following proposition provides a relationship between the second derivative of potential function and the adjoint twisted Reidemeister torsion. Similar phenomenon for once-punctured torus bundles has been observed in \cite{BWYIII}.
\begin{proposition}\label{Hess} We have
$$ f''_{1}(\alpha_1(\eta_v); \eta_v) f''_{2}(\alpha_2(\eta_v); \eta_v)
= \pm 4\mathrm{Tor}(M_\varphi, \mu_v,  \rho_{\theta_v^{(n)}}).$$
Moreover, $ f''_{1}(\alpha_1(\eta_v)) $ and $f''_{2}(\alpha_2(\eta_v))$ are non-zero.
\end{proposition}
\begin{proof}
Let $\tau(\Sigma, \mathcal{T}, \mu_v,  \rho_{\theta})$ be the 1-loop invariant defined in \cite{DG} with respect to the ideal triangulation $\mathcal{T}$, the meridian $\mu_v$ and the holonomy representation $\rho_{\theta_v}$. It is known that $\mathrm{Tor}(M_\varphi, \mu_v,  \rho_{\theta_v^{(n)}})$ coincides with $\tau(\Sigma, \mathcal{T}, \mu_v,  \rho_{\theta})$ \cite[Section 4.6]{DG}. We assign shape parameters to the ideal triangulation as shown in Figure \ref{boundary}. Consider the combinatorial flattening $(f_1, f_1', f_1'') = (f_2, f_2', f_2'')= (1,0,0)$.
By the definition of the 1-loop invariant, (\ref{Hm1}), (\ref{Hm2}) and (\ref{He}), we have
\begin{align*}
\tau(\Sigma, \mathcal{T}, \mu_v,  \rho_{\theta})
&= \pm \frac{\mathrm{det}
\begin{pmatrix}
 \frac{\partial}{\partial z_1} \mathrm{H}(e) &  \frac{\partial}{\partial z_2} \mathrm{H}(e) \\
  \frac{\partial}{\partial z_1} \mathrm{H}(\mu_2) &  \frac{\partial}{\partial z_1} \mathrm{H}(\mu_2)
\end{pmatrix}
}{2\prod_{i=1}^2 \xi_i  } \\
&= \pm \frac{z_1z_2}{4} \mathrm{det}\begin{pmatrix}
  \frac{\partial}{\partial z_1} \mathrm{H}(\mu_1) &  \frac{\partial}{\partial z_1} \mathrm{H}(\mu_1)\\
  \frac{\partial}{\partial z_1} \mathrm{H}(\mu_2) &  \frac{\partial}{\partial z_1} \mathrm{H}(\mu_2)
\end{pmatrix} \\
&= \pm \frac{(1+z_1)(1+z_2)}{(1-z_1)(1-z_2)}.
\end{align*}
Besides, by a direct computation,
\begin{align}
f''_{1}(\alpha_1(\eta_v); \eta_v)
&= \frac{4e^{-2\alpha_1(\eta_v)\sqrt{-1}}}{1+e^{-2\alpha_1(\eta_v)\sqrt{-1}}} - 2
= 2\Bigg(\frac{1+z_1}{1-z_1}\Bigg), \label{f''for1}\\
f''_{2}(\alpha_2(\eta_v); \eta_v)
&=  \frac{4e^{-2\alpha_2(\eta_v)\sqrt{-1}}}{1+e^{-2\alpha_2(\eta_v)\sqrt{-1}}} - 2
= 2\Bigg(\frac{1+z_2}{1-z_2}\Bigg), \label{f''for2}
\end{align}
where $z_1 = -e^{-2\sqrt{-1}\alpha_1(\eta_v)}$ and $z_2 = -e^{-2\sqrt{-1}\alpha_2(\eta_v)}$. Thus,
$$ f''_{1}(\alpha_1(\eta_v); \eta_v) f''_{2}(\alpha_2(\eta_v); \eta_v)
= \pm 4\tau(\Sigma, \mathcal{T}, \mu_v,  \rho_{\theta})
= \pm 4\mathrm{Tor}(M_\varphi, \mu_v,  \rho_{\theta}).
$$
Finally, by Lemma \ref{geotricone}, since the triangulation is geometric for all $\eta_v \in (-\pi,\pi)$, by (\ref{f''for1}) and (\ref{f''for2}), we have $f''_{1}(\alpha_1(\eta_v)), f''_{2}(\alpha_2(\eta_v))  \neq 0$. This completes the proof.
\end{proof}

\subsection{Neumann-Zagier potential function}\label{NZPF}
Consider the pair of generators $(\mu_v, l)$ of the fundamental group of the boundary torus of $M_\varphi$.
Recall from \cite{NZ} that there exists a holomorphic function $\Phi(\mathrm{H}(\mu_v))$ defined on a neighborhood of $0\in \CC$ satisfying
$$ \frac{d \Phi(\mathrm{H}(\mu_v))}{d \mathrm{H}(\mu_v)} = \frac{\mathrm{H}(l)}{2} \quad\text{and}\quad \Phi(0)= \sqrt{-1}(\mathrm{Vol}(M_\varphi)+\sqrt{-1}\mathrm{CS}(M_\varphi)),$$
where $\mathrm{Vol}(M_\varphi)$ and $\mathrm{CS}(M_\varphi)$ are the hyperbolic volume and the Chern-Simons invariant of $M_\varphi$ respectively. It is well-known that
$ \mathrm{Vol}(M_\varphi) = 2v_3$, where $v_3$ is the volume of a regular ideal tetrahedron. Besides, $\mathrm{CS}(M_\varphi)=0$ due to the fact that the figure eight knot is isotopic to its mirror image. We observe that the sum of critical values of $f_1$ and $f_2$ is related to the Neumann-Zagier potential function as follows.
\begin{proposition} Let $\mathrm{H}(\mu_v)= 2\eta_v \sqrt{-1}$. We have
$$\Phi(\mathrm{H}(\mu_v))
=f_1(\alpha_1(\eta_v); \eta_v) +  f_2(\alpha_2(\eta_v); \eta_v) .
$$
\end{proposition}
\begin{proof}
By Chain rule,
\begin{align*}
&\frac{\partial}{\partial \mathrm{H}(\mu_v)} (f_1(\alpha_1(\eta_v); \eta_v) +  f_2(\alpha_2(\eta_v); \eta_v))\\
=& \frac{1}{2\sqrt{-1}} \frac{\partial}{\partial \eta_v} (f_1(\alpha_1(\eta_v); \eta_v) +  f_2(\alpha_2(\eta_v); \eta_v))\\
=&\frac{1}{2\sqrt{-1}}\Bigg( \frac{\partial f_1}{\partial \alpha_1} (\alpha_1(\eta_v); \eta_v) + \frac{\partial f_2}{\partial \alpha_2} (\alpha_2(\eta_v); \eta_v)
+ \frac{\partial f_1}{\partial \eta_v} (\alpha_1(\eta_v); \eta_v) + \frac{\partial f_2}{\partial \eta_v} (\alpha_2(\eta_v); \eta_v) \Bigg)\\
=& \frac{1}{2\sqrt{-1}} \Bigg(\frac{\partial f_1}{\partial \eta_v} (\alpha_1(\eta_v); \eta_v) + \frac{\partial f_2}{\partial \eta_v} (\alpha_2(\eta_v); \eta_v) \Bigg) \\
=& \frac{1}{2\sqrt{-1}} ( \alpha_1(\eta_v) - \alpha_2(\eta_v) ) \\
=& \frac{1}{4} ( \log z_2- \log z_1 + 2k\pi\sqrt{-1}),
\end{align*}
where $k\in\ZZ$ and  the third equality follows from the facts that $\alpha_1(\eta_v)$ and $\alpha_2(\eta_v)$ are critical points of $f_1$ and $f_2$ respectively. When $\eta_v=0$, we have $\alpha_1(\eta_v) = \alpha_2(\eta_v) = \frac{\pi}{3}$. By continuity and Lemma \ref{geotricone}, we can conclude that $k=0$. Hence,
$$\frac{\partial}{\partial \mathrm{H}(\mu_v)}  (f_1(\alpha_1(\eta_v); \eta_v) +  f_2(\alpha_2(\eta_v); \eta_v))= \frac{1}{4} (\log z_2- \log z_1).$$
From Figure \ref{boundary}, the holonomy of the purple curve $2l$ is given by
$$ \mathrm{H}(2l)  = \log z_2 - \log z_1.$$
Altogether, we have
\begin{align*}
\frac{\partial}{\partial \mathrm{H}(\mu_v)} (f_1(\alpha_1(\eta_v); \eta_v) +  f_2(\alpha_2(\eta_v); \eta_v))
= \frac{\mathrm{H}(2l)}{4} = \frac{\mathrm{H}(l)}{2}.
\end{align*}

Finally, note that by Proposition \ref{2loba}, a direct computation shows that
$$\Phi(0)
= f_1\Big(\frac{\pi}{3};0 \Big)+ f_2\Big(\frac{\pi}{3};0\Big)
= \sqrt{-1}(\mathrm{Vol}(M_\varphi)+\sqrt{-1}\mathrm{CS}(M_\varphi)).$$
This completes the proof.
\end{proof}

\section{Asymptotics of the leading Fourier coefficients}\label{ALFC}

We first choose a sufficiently small $\delta>0$ for the estimation.
\begin{lemma}\label{assumpvol} For each $eta_v \in (-\pi,\pi)$, there exists a $\delta>0$ such that
\begin{align*}
2\Lambda(2\delta) <
\min\{\mathrm{Im} f_{1}(\alpha_1(\eta_v); \eta_v),\mathrm{Im} f_{2}(\alpha_2(\eta_v); \eta_v)\}.
\end{align*}
\end{lemma}
\begin{proof}
By Proposition \ref{geotricone}, we have
\begin{align*}
\min\{\mathrm{Im} f_{1}(\alpha_1(\eta_v); \eta_v),\mathrm{Im} f_{2}(\alpha_2(\eta_v); \eta_v)\}
&= \min\Big\{ D\Big(-e^{-2\alpha_1(\eta)\sqrt{-1}}\Big), D\Big(-e^{-2\alpha_1(\eta)\sqrt{-1}}\Big)\Big\}
> 0 .
\end{align*}
Since $\Lambda(0)=0$, the existence of $\delta>0$ follows from the continuity of $\Lambda$.
\end{proof}

From now on, we let $\delta>0$ be a positive number given in Lemma \ref{assumpvol}.
The following convexity result will be used later to estimate Fourier coefficients.

\begin{lemma}\label{convexity}
Let $\alpha_1 = x_1 + \sqrt{-1}y_1$, $\alpha_2 = x_2 + \sqrt{-1} y_2 $ for $ x_1,x_2,y_1,y_2 \in \RR$. Then,
\\
$\mathrm{Im} f_{1}(\alpha_1; \eta_v) $ and $\mathrm{Im} f_{2}(\alpha_2; \eta_v) $ are
\begin{enumerate}
\item concave up in $y_1$ and $y_2$ on  $0 \leq x_1, x_2 \leq \frac{\pi}{2}$ and $ \pi  \leq x_1,x_2 \leq \frac{3\pi}{2}$, and
\item concave down in $y_1$ and $y_2$ on $- \frac{\pi}{2} \leq x_1, x_2 \leq 0$ and $\frac{\pi}{2}\leq x_1,x_2 \leq \pi$.
\end{enumerate}
Moreover, $\mathrm{Im} f_{1}(\alpha_1; \eta_v) $ and $\mathrm{Im} f_{2}(\alpha_2; \eta_v) $ are
\begin{enumerate}
\item concave down in $x_1$ and $x_2$ on $0 \leq x_1, x_2 \leq \frac{\pi}{2}$ and $ \pi \leq x_1,x_2 \leq \frac{3\pi}{2}$, and
\item concave up in $x_1$ and $x_2$ on $- \frac{\pi}{2} \leq x_1, x_2 \leq 0$ and $\frac{\pi}{2}\leq x_1,x_2 \leq \pi$.
\end{enumerate}
\end{lemma}
\begin{proof}
By using the Cauchy-Riemann equation, we have
\begin{align*}
\frac{\partial}{\partial y_1}\mathrm{Im} f_{1}(\alpha_1; \eta_v) &= \mathrm{Re} \frac{\partial}{\partial \alpha_1} f_{1}(\alpha_1; \eta_v) = -2\mathrm{Arg}\Big(1+e^{-2x_1\sqrt{-1}+ 2y_1} \Big) - 2x_1 - \eta_v \\
\frac{\partial}{\partial y_2}\mathrm{Im} f_{2}(\alpha_2; \eta_v) &= \mathrm{Re} \frac{\partial}{\partial \alpha_2} f_{2}(\alpha_2; \eta_v) = -2\mathrm{Arg}\Big(1+e^{-2x_2\sqrt{-1}+ 2y_2} \Big) - 2x_2+ \eta_v.
\end{align*}
In particular, we have
\begin{align*}
\frac{\partial^2}{\partial y_1^2}\mathrm{Im} f_{1}(\alpha_1; \eta_v) & = \frac{2\sin(2x_1)}{\cosh(2y_1)+\cos(2x_1)}  \\
\frac{\partial^2}{\partial y_2^2}\mathrm{Im} f_{2}(\alpha_2; \eta_v) &=  \frac{2\sin(2x_2)}{\cosh(2y_2)+\cos(2x_2)}  .
\end{align*}
This proves the first claim. The second claim follows from the fact that $\mathrm{Im} f_{1}(\alpha_1; \eta_v), \mathrm{Im} f_{2}(\alpha_2; \eta_v)$ are harmonic functions.
\end{proof}

We will use the following version of saddle point approximation to obtain the asymptotics of the leading Fourier coefficients.
\begin{proposition}[Proposition 5.1, \cite{WY1}] \label{saddle}
Let $D_{\mathbf z}$ be a region in $\mathbb C^n$ and let $D_{\mathbf a}$ be a region in $\mathbb R^k.$ Let $f(\mathbf z,\mathbf a)$ and $g(\mathbf z,\mathbf a)$ be complex valued functions on $D_{\mathbf z}\times D_{\mathbf a}$  which are holomorphic in $\mathbf z$ and smooth in $\mathbf a.$ For each positive integer $r,$ let $f_r(\mathbf z,\mathbf a)$ be a complex valued function on $D_{\mathbf z}\times D_{\mathbf a}$ holomorphic in $\mathbf z$ and smooth in $\mathbf a.$
For a fixed $\mathbf a\in D_{\mathbf a},$ let $f^{\mathbf a},$ $g^{\mathbf a}$ and $f_r^{\mathbf a}$ be the holomorphic functions  on $D_{\mathbf z}$ defined by
$f^{\mathbf a}(\mathbf z)=f(\mathbf z,\mathbf a),$ $g^{\mathbf a}(\mathbf z)=g(\mathbf z,\mathbf a)$ and $f_r^{\mathbf a}(\mathbf z)=f_r(\mathbf z,\mathbf a).$ Suppose $\{\mathbf a_r\}$ is a convergent sequence in $D_{\mathbf a}$ with $\lim_r\mathbf a_r=\mathbf a_0,$ $f_r^{\mathbf a_r}$ is of the form
$$ f_r^{\mathbf a_r}(\mathbf z) = f^{\mathbf a_r}(\mathbf z) + \frac{\upsilon_r(\mathbf z,\mathbf a_r)}{r^2},$$
$\{S_r\}$ is a sequence of embedded real $n$-dimensional closed disks in $D_{\mathbf z}$ sharing the same boundary and converging to an embedded $n$-dimensional disk $S_0$, and $\mathbf c_r$ is a point on $S_r$ such that $\{\mathbf c_r\}$ is convergent  in $D_{\mathbf z}$ with $\lim_r\mathbf c_r=\mathbf c_0.$ If for each $r$
\begin{enumerate}[(1)]
\item $\mathbf c_r$ is a critical point of $f^{\mathbf a_r}$ in $D_{\mathbf z},$
\item $\mathrm{Re}f^{\mathbf a_r}(\mathbf c_r) > \mathrm{Re}f^{\mathbf a_r}(\mathbf z)$ for all $\mathbf z \in S_r\setminus \{\mathbf c_r\},$
\item the domain $\{\mathbf z\in D_{\mathbf z}\ |\ \mathrm{Re} f^{\mathbf a_r}(\mathbf z) < \mathrm{Re} f^{\mathbf a_r}(\mathbf c_r)\}$ deformation retracts to $S_r\setminus\{\mathbf c_r\},$
\item $|g^{\mathbf a_r}(\mathbf c_r)|$ is bounded from below by a positive constant independent of $r,$
\item $|\upsilon_r(\mathbf z, \mathbf a_r)|$ is bounded from above by a constant independent of $r$ on $D_{\mathbf z},$ and
\item  the Hessian matrix $\mathrm{Hess}(f^{\mathbf a_0})$ of $f^{\mathbf a_0}$ at $\mathbf c_0$ is non-singular,
\end{enumerate}
then
\begin{equation*}
\begin{split}
 \int_{S_r} g^{\mathbf a_r}(\mathbf z) e^{rf_r^{\mathbf a_r}(\mathbf z)} d\mathbf z= \Big(\frac{2\pi}{r}\Big)^{\frac{n}{2}}\frac{g^{\mathbf a_r}(\mathbf c_r)}{\sqrt{(-1)^n \det\mathrm{Hess}(f^{\mathbf a_r})(\mathbf c_r)}} e^{rf^{\mathbf a_r}(\mathbf c_r)} \Big( 1 + O \Big( \frac{1}{r} \Big) \Big).
 \end{split}
 \end{equation*}
\end{proposition}

\begin{proposition}\label{leading1}
\begin{align*}
\int_{C_{1,1}} g_1(\alpha_1) e^{\frac{n}{4\pi\sqrt{-1}} \Big(f_{1}(\alpha_1; \eta_v)
+ O(\frac{1}{n})\Big)} d\alpha_1
&= \frac{2\sqrt{2}\pi}{\sqrt{n}}\frac{g_1(\alpha_1(\eta_v))}{\sqrt{ \sqrt{-1} f_{1}''(\alpha_1(\eta_v))}} e^{\frac{n}{4\pi \sqrt{-1}} f_{1}(\alpha_1(\eta_v); \eta_v)}  \Big( 1 + O \Big( \frac{1}{n} \Big) \Big), \\
 \int_{C_{2,1}} g_2(\alpha_2) e^{\frac{n}{4\pi\sqrt{-1}} \Big(f_{2}(\alpha_2; \eta_v)+ O(\frac{1}{n})\Big)} d\alpha_2 &= \frac{2\sqrt{2}\pi}{\sqrt{n}} \frac{g_2(\alpha_2(\eta_v))}{\sqrt{ \sqrt{-1} f_{2}''(\alpha_2(\eta_v))}} e^{\frac{n}{4\pi \sqrt{-1}} f_{2}(\alpha_2(\eta_v); \eta_v)}  \Big( 1 + O \Big( \frac{1}{n} \Big) \Big).
\end{align*}
\end{proposition}
\begin{proof}
We first divide the contour $C_{1,1}$ into $C_{1,1} \cap \RR_{\leq 0}$ and $C_{1,1} \cap \RR_{\geq 0}$. For $\alpha \in C_{1,1} \cap \RR_{\leq 0}$, by Lemma \ref{assumpvol}, Proposition \ref{2loba}, Lemma \ref{imgivevol} and Lemma \ref{convexity},
$$ \mathrm{Im} f_{1}(\alpha_1; \eta_v) \leq \max\Big\{\mathrm{Im} f_{1}\Big( -\frac{\pi}{2}+\delta; \eta_v\Big), \mathrm{Im} f_{1}(0; \eta_v) \Big\} \leq 2\Lambda (2\delta) <  \mathrm{Im} f_{1}(\alpha_1(\eta_v); \eta_v) . $$
As a result,
$$  \int_{C_{1,1}\cap \RR_{\leq 0}} g_{1}(\alpha_1) e^{\frac{n}{4\pi\sqrt{-1}} \Big(f_{1}(\alpha_1; \eta_v) + O(\frac{1}{n})\Big)} d\alpha_1 = O\Big(e^{\frac{n}{2\pi}(\Lambda(2\delta))}\Big)
= o\Big(e^{\frac{n}{4\pi}(\mathrm{Im} f_{1}(\alpha_1(\eta_v); \eta_v))}\Big).$$
Next, consider the contour $C = C_{1,1}^{\text{top}} \cup C_{1,1}^{\text{side}}$ defined by
\begin{align*}
C_{1,1}^{\text{top}} &=  \Big\{ z=x+\sqrt{-1}y \mid 0 \leq x \leq \frac{\pi}{2}-\delta, y = \mathrm{Im} (\alpha(\eta_v))\Big\} \\
C_{1,1}^{\text{side}} &= \Big\{ z=  t\sqrt{-1}\mathrm{Im} (\alpha(\eta_v)) \mid t\in [0,1]\Big\}  \cup  \Big\{ z= \frac{\pi}{2} - \delta+ t\sqrt{-1}\mathrm{Im} (\alpha(\eta_v)) \mid t\in [0,1] \Big\}.
\end{align*}
By analyticity, we have
\begin{align}\label{integral1}
\int_{C_{1,1}\cap \RR_{\geq 0}} g_{1}(\alpha_1) e^{\frac{n}{4\pi\sqrt{-1}} \Big(f_{1}(\alpha_1; \eta_v) + O(\frac{1}{n})\Big)} d\alpha_1
=\int_{C} g_{1}(\alpha_1) e^{\frac{n}{4\pi\sqrt{-1}} \Big(f_{1}(\alpha_1; \eta_v)  + O(\frac{1}{n})\Big)} d\alpha_1
\end{align}

Now, we apply Proposition~\ref{saddle} to the integral in (\ref{integral1}). We check conditions (1)-(6) below:
\begin{enumerate}
\item By Lemma \ref{geotricone}, we have $\alpha (\eta_v) \in C_{1,1}^{\text{top}}$.
\item By Lemma \ref{convexity}, since $\mathrm{Im} f_{1}$ is concave down in $x_1$ and has a critical point at $\alpha(\eta_v)$, we have
$$\mathrm{Im} f_{1} (\alpha; \eta_v) < \mathrm{Im} f_{1} (\alpha(\eta_v); \eta_v)$$
for all $\alpha_1 \in C_{1,1,\text{top}}  \setminus\{\alpha(\eta_v)\}$. Besides, by Lemma \ref{convexity}, since $\mathrm{Im} f_{1}$ is concave up in $y_1$, we have
$$\mathrm{Im} f_{1} (\alpha; \eta_v) \leq \max \big\{ \mathrm{Im} f_{1} (0; \eta_v), \mathrm{Im} f_{1} \big( \sqrt{-1}\mathrm{Im} (\alpha(\eta_v)) ; \eta_v\big) \big\} < \mathrm{Im} f_{1} (\alpha(\eta_v); \eta_v)$$
for $\alpha \in \{ z=  t\sqrt{-1}\mathrm{Im} (\alpha(\eta_v)) \mid t\in [0,1]\} $ and
$$\mathrm{Im} f_{1} (\alpha; \eta_v) \leq \max \Big\{ \mathrm{Im} f_{1} \Big(\frac{\pi}{2}-\delta; \eta_v\Big), \mathrm{Im} f_{1} \Big(\frac{\pi}{2} - \delta + \sqrt{-1}\mathrm{Im} (\alpha(\eta_v)); \eta_v\Big) \Big\} < \mathrm{Im} f_{1} (\alpha(\eta_v); \eta_v)$$
for $\alpha \in \{ z= \frac{\pi}{2}-\delta+ t\sqrt{-1}\mathrm{Im} (\alpha(\eta_v)) \mid t\in [0,1]\} $.
\item By Lemma \ref{convexity},
$$ \Big\{ \alpha=x+\sqrt{-1}y \mid x\in \Big[0, \frac{\pi}{2}-\delta\Big], y\in \RR, \mathrm{Im}f_1(\alpha;\eta_v) < \mathrm{Im} f_{1} (\alpha(\eta_v); \eta_v)\Big\}$$
deformation retract to $C_{1,1}^{\text{top}}$.
\item By continuity and compactness of $C$,
\begin{align*}
|g_1(\alpha_1)|=\lt|  e^{-(\frac{\hat{l}_0 \sqrt{-1}}{2} + \frac{V}{2\pi}) \alpha_1 } (1+e^{-2\alpha_1\sqrt{-1}} )^{\frac{A_1}{4\pi \sqrt{-1}}} \rt|
\end{align*}
is non-zero and bounded below by a positive constant independent of $n$.

\item By Proposition \ref{conttoclass}, condition (5) holds.
\item By Proposition \ref{Hess}, condition (6) holds.
\end{enumerate}

Altogether, by Proposition \ref{saddle}, we have
\begin{align*}
&\int_{C_{1,1}} g_1(\alpha_1) e^{\frac{n}{4\pi\sqrt{-1}} \Big(f_{1}(\alpha_1; \eta_v)
+ O(\frac{1}{n})\Big)} d\alpha_1\\
=& \Big(\frac{2\pi}{n}\Big)^{\frac{1}{2}} \frac{g_{1}(\alpha_1(\eta_v))}{\sqrt{- \Big(\frac{f}{4\pi\sqrt{-1}}\Big)''(\alpha_1(\eta_v); \eta_v)}} e^{\frac{n}{4\pi \sqrt{-1}} f_{1}(\alpha_1(\eta_v); \eta_v)}  \Big( 1 + O \Big( \frac{1}{n} \Big) \Big).
\end{align*}
The proof for the second integral is similar.
\end{proof}

\begin{proposition}\label{leading2}
\begin{align*}
&\int_{C_{1,2}} g_1(\alpha_1) e^{\frac{n}{4\pi\sqrt{-1}} \Big(f_{1}(\alpha_1; \eta_v)+ 2\pi \alpha_1
+ O(\frac{1}{n})\Big)} d\alpha_1 \\
&=  \frac{2\sqrt{2}\pi}{\sqrt{n}} \frac{g_1(\pi+\alpha_1(\eta_v))}{\sqrt{ \sqrt{-1} f_{1}''(\pi+\alpha_1(\eta_v))}} e^{\frac{n}{4\pi \sqrt{-1}} f_{1}(\pi+\alpha_1(\eta_v); \eta_v)}  \Big( 1 + O \Big( \frac{1}{n} \Big) \Big), \\
& \int_{C_{2,2}} g_2(\alpha_2) e^{\frac{n}{4\pi\sqrt{-1}} \Big(f_{2}(\alpha_2; \eta_v)+ 2\pi \alpha_2 + O(\frac{1}{n})\Big)} d\alpha_2\\
 &=  \frac{2\sqrt{2}\pi}{\sqrt{n}} \frac{g_2(\pi+\alpha_2(\eta_v))}{\sqrt{\sqrt{-1} f_{2}''(\pi+\alpha_2(\eta_v))}} e^{\frac{n}{4\pi \sqrt{-1}} f_{2}(\pi+\alpha_2(\eta_v); \eta_v)}  \Big( 1 + O \Big( \frac{1}{n} \Big) \Big).
\end{align*}
\end{proposition}
\begin{proof}
We first divide the contour $C_{1,2}$ into $C_{1,2} \cap [\frac{\pi}{2}, \pi]$ and $C_{1,2} \cap [\pi, \frac{3\pi}{2}]$. For $\alpha \in C_{1,2} \cap [\frac{\pi}{2}, \pi]$, by Lemma \ref{convexity},
$$ \mathrm{Im} f_{1}(\alpha_1; \eta_v) \leq \max\Big\{f_{1}\Big( \frac{\pi}{2}+\delta ; \eta_v\Big), f_{1}(\pi;\eta_v) \Big\} \leq 2\Lambda (2\delta) < \mathrm{Im} f_{1}(\alpha_1(\eta_v); \eta_v) . $$
As a result,
$$  \int_{C_{1,2}\cap [\frac{\pi}{2},\pi]} g_{1}(\alpha_1) e^{\frac{n}{4\pi\sqrt{-1}} \Big(f_{1}(\alpha_1; \eta_v) + O(\frac{1}{n})\Big)} d\alpha_1 = O\Big(e^{\frac{n}{2\pi}(\Lambda(2\delta))}\Big)
= o\Big(e^{\frac{n}{4\pi}(\mathrm{Im} f_{1}(\alpha_1(\eta_v); \eta_v))}\Big).$$
Next, consider the contour $C = C_{1,2}^{\text{top}} \cup C_{1,2}^{\text{side}}$ defined by
\begin{align*}
C_{1,2}^{\text{top}} &=  \Big\{ z=x+\sqrt{-1}y \mid \pi \leq x \leq \frac{3\pi}{2}-\delta, y = \mathrm{Im} (\alpha(\eta_v))\Big\} \\
C_{1,2}^{\text{side}} &= \Big\{ z= \pi+ t\sqrt{-1}\mathrm{Im} (\alpha(\eta_v)) \mid t\in [0,1]\Big\}  \cup  \Big\{ z= \frac{3\pi}{2} - \delta+ t\sqrt{-1}\mathrm{Im} (\alpha(\eta_v)) \mid t\in [0,1] \Big\}.
\end{align*}
By analyticity, we have
\begin{align}\label{integral2}
\int_{C_{1,2}\cap [\pi,\frac{3\pi}{2}]} g_{1}(\alpha_1) e^{\frac{n}{4\pi\sqrt{-1}} \Big(f_{1}(\alpha_1; \eta_v) + 2\pi\alpha_1 + O(\frac{1}{n})\Big)} d\alpha_1
=\int_{C} g_{1}(\alpha_1) e^{\frac{n}{4\pi\sqrt{-1}} \Big(f_{1}(\alpha_1; \eta_v) + 2\pi\alpha_1 + O(\frac{1}{n})\Big)} d\alpha_1
\end{align}
A direct computation shows that $\alpha (\eta_v) + \pi $ is a critical point of the function $f_{1}(\alpha_1; \eta_v) + 2\pi\alpha_1$ with critical value
$$ f_{1}(\alpha_1(\eta_v) + \pi ; \eta_v) + 2\pi(\alpha_1(\eta_v) + \pi )
= f_{1}(\alpha_1(\eta_v) ; \eta_v) + \pi^2 + 2\pi \eta_v.
$$
In particular, we have
$$ \mathrm{Im}\big( f_{1}(\alpha_1(\eta_v) + \pi ; \eta_v) + 2\pi(\alpha_1(\eta_v) + \pi ) \big)
=  \mathrm{Im}(f_{1}(\alpha_1(\eta_v) ; \eta_v) .$$

Now, we apply Proposition~\ref{saddle} to the integral in (\ref{integral2}). We check the conditions (1)-(6) below:
\begin{enumerate}
\item By Lemma \ref{geotricone}, we have $\alpha (\eta_v) + \pi \in C_{1,2}^{\text{top}}$. By a direct computation, $\alpha (\eta_v) + \pi $ is a critical point of the function $f_{1}(\alpha_1; \eta_v) + 2\pi\alpha_1$.
\item By Lemma \ref{convexity}, since $\mathrm{Im} (f_{1}(\alpha_1; \eta_v) + 2\pi\alpha_1)$ is concave down in $x_1$ and has a critical point at $\pi+\alpha(\eta_v)$, we have
$$\mathrm{Im} (f_{1}(\alpha_1; \eta_v) + 2\pi\alpha_1) < \mathrm{Im}\big( f_{1}(\alpha_1(\eta_v) + \pi ; \eta_v) + 2\pi(\alpha_1(\eta_v) + \pi ) \big)= \mathrm{Im}(f_{1}(\alpha_1(\eta_v) ; \eta_v)$$
for all $\alpha_1 \in C_{1,2}^{\text{top}}  \setminus\{\alpha(\eta_v)+\pi\}$. Besides, by Lemma \ref{convexity}, since $\mathrm{Im} f_{1}$ is concave up in $y_1$, we have
\begin{align*}
&\mathrm{Im} (f_{1}(\alpha_1; \eta_v) + 2\pi\alpha_1) \\
<& \max \Big\{ \mathrm{Im} \Big( f_{1} \Big(\pi; \eta_v\Big) + 2\pi \Big(\pi\Big) \Big), \\
&\qquad\qquad\mathrm{Im} \Big( f_{1} \Big(\pi + \sqrt{-1}\mathrm{Im} (\alpha(\eta_v)); \eta_v\Big) + 2\pi \Big(\pi + \sqrt{-1}\mathrm{Im} (\alpha(\eta_v))\Big) \Big\}\\
<&  \mathrm{Im}(f_{1}(\alpha_1(\eta_v) ; \eta_v)
\end{align*}
for $\alpha \in \{ z=  \pi + t\sqrt{-1}\mathrm{Im} (\alpha(\eta_v)) \mid t\in [0,1]\} $ and
\begin{align*}
&\mathrm{Im} (f_{1}(\alpha_1; \eta_v) + 2\pi\alpha_1) \\
<& \max \Big\{ \mathrm{Im} \Big( f_{1} \Big(\frac{3\pi}{2}-\delta; \eta_v\Big) + 2\pi \Big(\frac{3\pi}{2}-\delta\Big) \Big), \\
&\qquad\qquad\mathrm{Im} \Big( f_{1} \Big(\frac{3\pi}{2}-\delta + \sqrt{-1}\mathrm{Im} (\alpha(\eta_v)); \eta_v\Big) + 2\pi \Big(\frac{3\pi}{2}-\delta + \sqrt{-1}\mathrm{Im} (\alpha(\eta_v))\Big) \Big\}\\
<&  \mathrm{Im}(f_{1}(\alpha_1(\eta_v) ; \eta_v)
\end{align*}
for $\alpha \in \{ z= \frac{3\pi}{2}-\delta+ t\sqrt{-1}\mathrm{Im} (\alpha(\eta_v)) \mid t\in [0,1]\} $.
\item By Lemma \ref{convexity},
$$ \Bigg\{ \alpha=x+\sqrt{-1}y \text{ }\Big\vert\text{ } x\in \Big[\pi, \frac{3\pi}{2}-\delta\Big], y\in \RR, \mathrm{Im}f_1(\alpha;\eta_v) < \mathrm{Im} f_{1} (\alpha(\eta_v); \eta_v)\Bigg\}$$
deformation retract to $C_{1,2}^{\text{top}}$.
\item By continuity and compactness of $C$,
\begin{align*}
|g_1(\alpha_1)|=  \lt|  e^{-(\frac{\hat{l}_0 \sqrt{-1}}{2} + \frac{V}{2\pi}) \alpha_1} (1+e^{-2\alpha_1\sqrt{-1}} )^{\frac{A_1}{4\pi \sqrt{-1}}} \rt|
\end{align*}
is non-zero and bounded below by a positive constant independent of $n$.

\item By Proposition \ref{conttoclass}, condition (5) holds.
\item By Proposition \ref{Hess}, condition (6) holds.
\end{enumerate}

Altogether, by Proposition \ref{saddle}, we have
\begin{align*}
&\int_{C_{1,2}} g_1(\alpha_1) e^{\frac{n}{4\pi\sqrt{-1}} \Big(f_{1}(\alpha_1; \eta_v)+ 2\pi \alpha_1
+ O(\frac{1}{n})\Big)} d\alpha_1 \\
&= \Big(\frac{2\pi}{n}\Big)^{\frac{1}{2}} \frac{g_1(\pi+\alpha_1(\eta_v))}{\sqrt{- \Big(\frac{f_{1}}{4\pi\sqrt{-1}}\Big)''(\pi+\alpha_1(\eta_v))}} e^{\frac{n}{4\pi \sqrt{-1}} f_{1}(\pi+\alpha_1(\eta_v); \eta_v)}  \Big( 1 + O \Big( \frac{1}{n} \Big) \Big).
\end{align*}
The proof for the second integral is similar.
\end{proof}

\section{Estimation of other Fourier coefficients}\label{EOFC}

\begin{proposition}\label{OtherFest1}
For $k_1, k_2 \neq 0$, we have
\begin{align*}
\int_{C_{1,1}} g_1(\alpha_1) e^{\frac{n}{4\pi\sqrt{-1}} \Big(f_{1}(\alpha_1; \eta_v)
-4k_1\pi  \alpha_1 + O(\frac{1}{n})\Big)} d\alpha_1 &= O\Big(e^{\frac{n}{2\pi}\Lambda(2\delta)}\Big),\\
\int_{C_{2,1}} g_2(\alpha_2) e^{\frac{n}{4\pi\sqrt{-1}} \Big(f_{2}(\alpha_2; \eta_v)-4k_2\pi  \alpha_2 + O(\frac{1}{n})\Big)} d\alpha_2 &= O\Big(e^{\frac{n}{2\pi}\Lambda(2\delta)}\Big).
\end{align*}

\end{proposition}
\begin{proof}
For $\alpha_1 \in R_{1,1}$, we have $-\pi\leq -2x \leq \pi$. Since $\eta_v \in (-\pi,\pi)$ and
$$ \frac{\partial}{\partial y_1}\mathrm{Im} f_{1}(\alpha_1; \eta_v) = \mathrm{Re} \frac{\partial}{\partial \alpha_1} f_{1}(\alpha_1; \eta_v) = -2\mathrm{Arg}\Big(1+e^{-2x_1\sqrt{-1}+ 2y_1} \Big) - 2x_1 + \eta_v ,$$
we have
$$ -3\pi + \eta_v < \frac{\partial}{\partial y_1}\mathrm{Im} f_{1}(\alpha_1; \eta_v) < 3\pi + \eta_v.$$
In particular, for any $\alpha_1 \in R_{1,1}$,
$$ \frac{\partial}{\partial y_1}\mathrm{Im} \Big(f_{1}(\alpha_1; \eta_v) - 4k\pi \sqrt{-1}\alpha_1\Big)
< -\pi + \eta_v < 0 \quad \text{if $k\geq 1$}$$
and
$$ \frac{\partial}{\partial y_1}\mathrm{Im} \Big(f_{1}(\alpha_1; \eta_v) - 4k\pi \sqrt{-1}\alpha_1\Big)
> \pi + \eta_v > 0  \quad \text{if $k\leq -1$.}
$$
For any $L\in \RR$, consider the contour $C_{1,1,L} = C_{1,1,L}^{\text{top}} \cup C_{1,1,L}^{\text{side}}$ defined by
\begin{align*}
C_{1,1,L}^{\text{top}} &=  \Big\{ z=x+\sqrt{-1}L \mid 0 \leq x \leq \frac{\pi}{2}-\delta\Big\}, \\
C_{1,1,L}^{\text{side}} &= \Big\{ z=  tL\sqrt{-1} \mid t\in [0,1]\Big\}  \cup  \Big\{ z= \frac{\pi}{2} - \delta+ tL\sqrt{-1} \mid t\in [0,1] \Big\}.
\end{align*}
As a result, if $k\geq 1$, by choosing sufficiently large $L>0$, we can make sure that the $\mathrm{Im} (f_{1}(\alpha_1; \eta_v) - 4k\pi \sqrt{-1}\alpha_1) < 2\Lambda(2\delta)$ on the top $C_{1,1,L}^{\text{top}}$. By convexity, we know that $\mathrm{Im} (f_{1}(\alpha_1; \eta_v) - 4k\pi \sqrt{-1}\alpha_1) < 2\Lambda(2\delta)$ on $C_{1,1,L}^{\text{side}}$. Similarly, if $k\leq -1$, by choosing $L<0$ with $|L|$ sufficiently large, we can make sure that the $\mathrm{Im} (f_{1}(\alpha_1; \eta_v) - 4k\pi \sqrt{-1}\alpha_1) < 2\Lambda(2\delta)$ on the top $C_{1,1,L}^{\text{top}}$. By convexity, we know that $\mathrm{Im} (f_{1}(\alpha_1; \eta_v) - 4k\pi \sqrt{-1}\alpha_1)< 2\Lambda(2\delta)$ on the side $C_{1,1,L}^{\text{side}}$. This completes the proof of the proposition for the first estimation.

For the second estimation, for $\alpha_2 \in R_{2,1}$, since $\eta_v \in (-\pi,\pi)$ and
$$ \frac{\partial}{\partial y_1}\mathrm{Im} f_{2}(\alpha_2; \eta_v) = \mathrm{Re} \frac{\partial}{\partial \alpha_2} f_{2}(\alpha_2; \eta_v) = -2\mathrm{Arg}(1+e^{-2x_2\sqrt{-1}+ 2y_2} ) - 2x_2 - \eta_v ,$$
we have
$$ -3\pi - \eta_v < \frac{\partial}{\partial y_1}\mathrm{Im} f_{1}(\alpha_1; \eta_v) < 3\pi - \eta_v.$$
The proof is similar to the previous case.
\end{proof}

\begin{proposition}\label{OtherFest2}
For $k_1, k_2  \neq -1$, we have
\begin{align*}
\int_{C_{1,2}} g_1(\alpha_1) e^{\frac{n}{4\pi\sqrt{-1}} \Big(f_{1}(\alpha_1; \eta_v)
-4k_1\pi  \alpha_1 + O(\frac{1}{n})\Big)} d\alpha_1 &= O\Big(e^{\frac{n}{2\pi}\Lambda(2\delta)}\Big),\\
\int_{C_{2,2}} g_2(\alpha_2) e^{\frac{n}{4\pi\sqrt{-1}} \Big(f_{2}(\alpha_2; \eta_v)-4k_2\pi  \alpha_2 + O(\frac{1}{n})\Big)} d\alpha_2 &= O\Big(e^{\frac{n}{2\pi}\Lambda(2\delta)}\Big).
\end{align*}
\end{proposition}
\begin{proof} The proof is similar to that for Proposition \ref{OtherFest1}.
We first divide the contour $C_{1,2}$ into $C_{1,2} \cap [\frac{\pi}{2}, \pi]$ and $C_{1,2} \cap [\pi, \frac{3\pi}{2}]$. On $C_{1,2} \cap [\frac{\pi}{2}, \pi]$, by convexity, we have
$$ \mathrm{Im}\Big(f_{1}(\alpha_1; \eta_v) - 4k\pi \sqrt{-1}\alpha_1 \Big)
=\mathrm{Im} f_{1}(\alpha_1; \eta_v)
\leq \max\Big\{f_{1}\Big( \frac{\pi}{2}+\delta ; \eta_v\Big), f_{1}(\pi;\eta_v) \Big\} \leq 2\Lambda (2\delta)  . $$
Thus,
$$ \int_{C_{1,2}\cap [\frac{\pi}{2}, \pi]} g_1(\alpha_1) e^{\frac{n}{4\pi\sqrt{-1}} \Big(f_{1}(\alpha_1; \eta_v)+ O(\frac{1}{n})\Big)} d\alpha_1 = O\Big(e^{\frac{n}{2\pi}\Lambda(2\delta)}\Big).$$
For $\alpha_1 \in R_{1,2} \cap ([\frac{\pi}{2}, \pi] \times \RR)$, we have $-3\pi < -2x < -2\pi$ and hence $-2\pi < -2\mathrm{Arg}(1+e^{-2x_1\sqrt{-1}+ 2y_1} )  < 0$. Since $\eta_v \in (-\pi,\pi)$ and
$$ \frac{\partial}{\partial y_1}\mathrm{Im} f_{1}(\alpha_1; \eta_v) = \mathrm{Re} \frac{\partial}{\partial \alpha_1} f_{1}(\alpha_1; \eta_v) = -2\mathrm{Arg}(1+e^{-2x_1\sqrt{-1}+ 2y_1} ) - 2x_1 + \eta_v ,$$
we have
$$ -5\pi + \eta_v < \frac{\partial}{\partial y_1}\mathrm{Im} f_{1}(\alpha_1; \eta_v) < -2\pi + \eta_v.$$
In particular, for any $\alpha_1 \in R_{1,2} \cap ([\frac{\pi}{2}, \pi] \times \RR)$,
$$ \frac{\partial}{\partial y_1}\mathrm{Im} (f_{1}(\alpha_1; \eta_v) - 4k_1\pi \sqrt{-1}\alpha_1)
< -2\pi + \eta_v < 0 \quad \text{if $k\geq 0$}$$
and
$$ \frac{\partial}{\partial y_1}\mathrm{Im} (f_{1}(\alpha_1; \eta_v) - 4k\pi \sqrt{-1}\alpha_1)
> 3\pi + \eta_v > 0  \quad \text{if $k\leq -2$}
$$
As a result, if $k\geq 0$, by pushing the contour upward in $y_1$ direction, we can make sure that the $\mathrm{Im} (f_{1}(\alpha_1; \eta_v)-4k\pi\sqrt{-1}) < \delta$ on the top. By convexity, we know that $\mathrm{Im} (f_{1}(\alpha_1; \eta_v)-4k\pi\sqrt{-1}) < \delta$ on the side. Similarly, if $k\leq -2$, by pushing the contour downward in $y_1$ direction, we can make sure that the $\mathrm{Im} (f_{1}(\alpha_1; \eta_v)-4k\pi\sqrt{-1}) < \delta$ on the top. By convexity, we know that $\mathrm{Im} (f_{1}(\alpha_1; \eta_v) -4k\pi\sqrt{-1})< \delta$ on the side. This completes the proof of the first estimation. The proof of the second estimation is similar.
\end{proof}

\begin{proposition}\label{OtherFest3}
We have
\begin{align*}
\int_{C_{1,2}} g_1(\alpha_1) e^{\frac{n}{4\pi\sqrt{-1}} \Big(f_{1}(\alpha_1; \eta_v)
+ 4\pi  \alpha_1 + O(\frac{1}{n})\Big)} d\alpha_1 &= O\Big(e^{\frac{n}{2\pi}\Lambda(2\delta)}\Big),\\
\int_{C_{2,2}} g_2(\alpha_2) e^{\frac{n}{4\pi\sqrt{-1}} \Big(f_{2}(\alpha_2; \eta_v)+ 4\pi  \alpha_2 + O(\frac{1}{n})\Big)} d\alpha_2 &= O\Big(e^{\frac{n}{2\pi}\Lambda(2\delta)}\Big).
\end{align*}
\end{proposition}
\begin{proof}
We first divide the contour $C_{1,2}$ into $C_{1,2} \cap [\frac{\pi}{2}, \pi]$ and $C_{1,2} \cap [\pi, \frac{3\pi}{2}]$. On $C_{1,2} \cap [\frac{\pi}{2}, \pi]$, by convexity, we have we have
$$\mathrm{Im} f_{1}(\alpha_1; \eta_v)
\leq \max\Big\{f_{1}\Big( \frac{\pi}{2}+\delta ; \eta_v\Big), f_{1}(\pi;\eta_v) \Big\} \leq 2\Lambda (2\delta)  . $$
Thus,
$$ \int_{C_{1,2}\cap [\frac{\pi}{2}, \pi]} g_1(\alpha_1) e^{\frac{n}{4\pi\sqrt{-1}} \Big(f_{1}(\alpha_1; \eta_v)+ O(\frac{1}{n})\Big)} d\alpha_1 = O\Big(e^{\frac{n}{2\pi}\Lambda(2\delta)}\Big).$$
For any $\alpha_1 =x+ \sqrt{-1}y \in R_{1,2}$ with $x\in [\frac{\pi}{2}, \pi]$, we have $-3\pi +\delta< -2x < -2\pi - \delta$. In particular,
\begin{align*}
\lim_{y_1 \to -\infty} \frac{\partial}{\partial y_1} \mathrm{Im} \Big( f_{1}(\alpha_1; \eta_v) +4\pi \alpha\Big)&=\lim_{y_1 \to -\infty} \Big[ -2\mathrm{Arg}\Big(1+e^{-2x_1\sqrt{-1}+ 2y_1} \Big) - 2x_1 + \eta_v + 4\pi \Big]\\
&= -2x_1 + \eta_v + 4\pi >\pi + \eta_v + \delta > 0.
\end{align*}
As a result, by pushing the contour downward in $y_1$ direction, we can make sure that the $\mathrm{Im} (f_{1}(\alpha_1; \eta_v) + 4\pi \alpha_1) < 2\Lambda(2\delta)$ on the top. By convexity, we know that $\mathrm{Im} (f_{1}(\alpha_1; \eta_v)+ 4\pi\alpha_1) < 2\Lambda(2\delta)$ on the side. This proves the first estimation. The proof of the second estimation is similar.
\end{proof}

We have the following analogue of Proposition \ref{OtherFest1}, \ref{OtherFest2} and \ref{OtherFest3}. The proofs are similar.

\begin{proposition}\label{OtherFest1'}
For $k_1, k_2 \neq 0$, we have
\begin{align*}
\int_{C_{1,2}} g_1(\alpha_1) e^{\frac{n}{4\pi\sqrt{-1}} \Big(f_{1}(\alpha_1; \eta_v)+ 2\pi \alpha_1
-4k_1\pi  \alpha_1 + O(\frac{1}{n})\Big)} d\alpha_1 &= O\Big(e^{\frac{n}{2\pi}\Lambda(2\delta)}\Big),\\
\int_{C_{2,2}} g_2(\alpha_2) e^{\frac{n}{4\pi\sqrt{-1}} \Big(f_{2}(\alpha_2; \eta_v)+ 2\pi \alpha_2 -4k_2\pi  \alpha_2 + O(\frac{1}{n})\Big)} d\alpha_2 &= O\Big(e^{\frac{n}{2\pi}\Lambda(2\delta)}\Big).
\end{align*}
\end{proposition}

\begin{proof}
For $\alpha_1 \in R_{1,2}$, we have $-\pi\leq -2(x-\pi) \leq \pi$. Since $\eta_v \in (-\pi,\pi)$ and
$$ \frac{\partial}{\partial y_1}\mathrm{Im} \Big(f_{1}(\alpha_1; \eta_v)-2\pi\alpha_1\Big)  = -2\mathrm{Arg}(1+e^{-2x_1\sqrt{-1}+ 2y_1} ) - 2(x_1-\pi) + \eta_v ,$$
we have
$$ -3\pi + \eta_v < \frac{\partial}{\partial y_1}\mathrm{Im}\Big(f_{1}(\alpha_1; \eta_v)-2\pi\alpha_1\Big)  < 3\pi + \eta_v.$$
In particular, for any $\alpha_1 \in R_{1,2}$,
$$ \frac{\partial}{\partial y_1}\mathrm{Im} \Big(f_{1}(\alpha_1; \eta_v) -2\pi\alpha_1 - 4k\pi \sqrt{-1}\alpha_1\Big)
< -\pi + \eta_v < 0 \quad \text{if $k\geq 1$}$$
and
$$ \frac{\partial}{\partial y_1}\mathrm{Im} \Big(f_{1}(\alpha_1; \eta_v) -2\pi\alpha_1 - 4k\pi \sqrt{-1}\alpha_1\Big)
> \pi + \eta_v > 0  \quad \text{if $k\leq -1$}
$$
As a result, if $k\geq 1$, by pushing the contour upward in $y_1$ direction, we can make sure that the $\mathrm{Im} (f_{1}(\alpha_1; \eta_v) -2\pi\alpha_1 - 4k\pi \sqrt{-1}\alpha_1) < 2\Lambda(2\delta)$ on the top. By convexity, we know that $\mathrm{Im} (f_{1}(\alpha_1; \eta_v) -2\pi\alpha_1 - 4k\pi \sqrt{-1}\alpha_1) <2\Lambda(2\delta)$ on the side. Similarly, if $k\leq -1$, by pushing the contour downward in $y_1$ direction, we can make sure that the $\mathrm{Im} (f_{1}(\alpha_1; \eta_v) -2\pi\alpha_1 - 4k\pi \sqrt{-1}\alpha_1) <2\Lambda(2\delta)$ on the top. By convexity, we know that $\mathrm{Im} (f_{1}(\alpha_1; \eta_v) -2\pi\alpha_1 - 4k\pi \sqrt{-1}\alpha_1)<2\Lambda(2\delta)$ on the side. This completes the proof of the first claim. The proof of the second claim is similar.
\end{proof}

\begin{proposition}\label{OtherFest2'}
For $k_1, k_2  \neq 0$, we have
\begin{align*}
\int_{C_{1,1}} g_1(\alpha_1) e^{\frac{n}{4\pi\sqrt{-1}} \Big(f_{1}(\alpha_1; \eta_v)+ 2\pi\alpha_1
-4k_1\pi  \alpha_1 + O(\frac{1}{n})\Big)} d\alpha_1 &= O\Big(e^{\frac{n}{2\pi}\Lambda(2\delta)}\Big),\\
\int_{C_{2,1}} g_2(\alpha_2) e^{\frac{n}{4\pi\sqrt{-1}} \Big(f_{2}(\alpha_2; \eta_v)+ 2\pi\alpha_2 -4k_2\pi  \alpha_2 + O(\frac{1}{n})\Big)} d\alpha_2 &= O\Big(e^{\frac{n}{2\pi}\Lambda(2\delta)}\Big).
\end{align*}
\end{proposition}

\begin{proof}
We first divide the contour $C_{1,1}$ into $C_{1,1} \cap [-\frac{\pi}{2}, 0]$ and $C_{1,1} \cap [0, \frac{\pi}{2}]$. On $C_{1,1} \cap [-\frac{\pi}{2}, 0]$, by convexity, we have
$$ \int_{C_{1,1} \cap [-\frac{\pi}{2}, 0]} g_1(\alpha_1) e^{\frac{n}{4\pi\sqrt{-1}} \Big(f_{1}(\alpha_1; \eta_v)+ 2\pi\alpha_1 O(\frac{1}{n})\Big)} d\alpha_1 = O\Big(e^{\frac{n}{2\pi}\Lambda(2\delta)}\Big).$$
For $\alpha_1 \in R_{1,1} \cap ([0, \frac{\pi}{2}] \times \RR)$, we have $\pi < -2x + 2\pi < 2\pi$ and
$$-2\pi < -2\mathrm{Arg}(1+e^{-2x_1\sqrt{-1}+ 2y_1} )  < 0.$$
Since $\eta_v \in (-\pi,\pi)$ and
$$ \frac{\partial}{\partial y_1}\mathrm{Im} \Big( f_{1}(\alpha_1; \eta_v) +2\pi\alpha_1\Big) = -2\mathrm{Arg}(1+e^{-2x_1\sqrt{-1}+ 2y_1} ) - 2x_1 + 2\pi + \eta_v ,$$
we have
$$ -\pi + \eta_v < \frac{\partial}{\partial y_1}\mathrm{Im} f_{1}(\alpha_1; \eta_v) < 2\pi + \eta_v.$$
In particular, for any $\alpha_1 \in R_{1,1} \cap ([0, \frac{\pi}{2}] \times \RR)$,
$$ \frac{\partial}{\partial y_1}\mathrm{Im} (f_{1}(\alpha_1; \eta_v) + 2\pi\alpha_1 - 4k_1\pi \sqrt{-1}\alpha_1)
< -2\pi + \eta_v < 0 \quad \text{if $k\geq 1$}$$
and
$$ \frac{\partial}{\partial y_1}\mathrm{Im} (f_{1}(\alpha_1; \eta_v) - 4k\pi \sqrt{-1}\alpha_1)
> 3\pi + \eta_v > 0  \quad \text{if $k\leq -1$}
$$
As a result, if $k\geq 1$, by pushing the contour upward in $y_1$ direction, we can make sure that the $\mathrm{Im} (f_{1}(\alpha_1; \eta_v) + 2\pi\alpha_1 - 4k_1\pi \sqrt{-1}\alpha_1)
 <2\Lambda(2\delta)$ on the top. By convexity, we know that $\mathrm{Im} (f_{1}(\alpha_1; \eta_v) + 2\pi\alpha_1 - 4k_1\pi \sqrt{-1}\alpha_1)
 <2\Lambda(2\delta)$ on the side. Similarly, if $k\leq -1$, by pushing the contour downward in $y_1$ direction, we can make sure that the $\mathrm{Im} (f_{1}(\alpha_1; \eta_v) + 2\pi\alpha_1 - 4k_1\pi \sqrt{-1}\alpha_1)
 <2\Lambda(2\delta)$ on the top. By convexity, we know that $\mathrm{Im} (f_{1}(\alpha_1; \eta_v) + 2\pi\alpha_1 - 4k_1\pi \sqrt{-1}\alpha_1)
<2\Lambda(2\delta)$ on the side. This proves the first estimation. The proof of the second estimation is similar.
\end{proof}

\begin{proposition}\label{OtherFest3'}
We have
\begin{align*}
\int_{C_{1,2}} g_1(\alpha_1) e^{\frac{n}{4\pi\sqrt{-1}} \Big(f_{1}(\alpha_1; \eta_v)
+ 2\pi  \alpha_1 + O(\frac{1}{n})\Big)} d\alpha_1 &= O\Big(e^{\frac{n}{2\pi}\Lambda(2\delta)}\Big),\\
\int_{C_{2,2}} g_2(\alpha_2) e^{\frac{n}{4\pi\sqrt{-1}} \Big(f_{2}(\alpha_2; \eta_v)+ 4\pi  \alpha_2 + O(\frac{1}{n})\Big)} d\alpha_2 &= O\Big(e^{\frac{n}{2\pi}\Lambda(2\delta)}\Big).
\end{align*}
\end{proposition}

\begin{proof}
We first divide the contour $C_{1,1}$ into $C_{1,1} \cap [-\frac{\pi}{2}, 0]$ and $C_{1,1} \cap [0, \frac{\pi}{2}]$. On $C_{1,1} \cap [-\frac{\pi}{2}, 0]$, by convexity, we have
$$ \int_{C_{1,1} \cap [-\frac{\pi}{2}, 0]} g_1(\alpha_1) e^{\frac{n}{4\pi\sqrt{-1}} \Big(f_{1}(\alpha_1; \eta_v)+ 2\pi\alpha_1 + O(\frac{1}{n})\Big)} d\alpha_1 = O\Big(e^{\frac{n}{2\pi}\Lambda(2\delta)}\Big).$$
For $\alpha_1 = x+\sqrt{-1}y \in R_{1,1}$ with $x\in [0, \frac{\pi}{2}]$, we have $\pi < -2x + 2\pi < 2\pi$ and
\begin{align*}
\lim_{y_1 \to -\infty} \frac{\partial}{\partial y_1} \mathrm{Im} \Big( f_{1}(\alpha_1; \eta_v) +2\pi \alpha\Big)&=\lim_{y_1 \to -\infty} \Big[ -2\mathrm{Arg}(1+e^{-2x_1\sqrt{-1}+ 2y_1} ) - 2x_1 + 2\pi + \eta_v  \Big]\\
&= -2x_1 + 2\pi + \eta_v >\pi + \eta_v + \delta > 0.
\end{align*}
As a result, by pushing the contour downward in $y_1$ direction, we can make sure that the $\mathrm{Im} (f_{1}(\alpha_1; \eta_v) + 2\pi \alpha_1) <2\Lambda(2\delta)$ on the top. By convexity, we know that $\mathrm{Im} (f_{1}(\alpha_1; \eta_v)+ 2\pi \alpha_1) <2\Lambda(2\delta)$ on the side. This proves the first estimation. The proof of the second estimation is similar.
\end{proof}

\section{Asymptotics of the invariant}\label{AOI}
\begin{proposition}\label{asym4} We have the following asymptotic expansion formulas.
\begin{enumerate}
\item If $p_1=0$, then
$$ \Sigma_{1} = \sqrt{2n} \frac{g_1(\alpha_1(\eta_v);\eta_v)}{\sqrt{- f_{1}''(\alpha_1(\eta_v);\eta_v)}} e^{\frac{n}{4\pi \sqrt{-1}} f_{1}(\alpha_1(\eta_v);\eta_v)} \Big(1 + O\Big(\frac{1}{n}\Big) \Big)$$
\item If $p_1=1$, then
$$ \Sigma_{1} = \sqrt{2n} \Big(1 +  (\sqrt{-1})^n e^{\frac{A_1}{2}} \Big) \frac{g_1(\pi+\alpha_1(\eta_v);\eta_v)}{\sqrt{- f_{1}''(\pi+\alpha_1(\eta_v);\eta_v)}} e^{\frac{n}{4\pi \sqrt{-1}} f_{1}(\pi+\alpha_1(\eta_v);\eta_v)}  \Big( 1 + O \Big( \frac{1}{n} \Big) \Big)$$
\item If $p_2=0$, then
$$ \Sigma_{2} = \sqrt{2n} \frac{g_2(\alpha_2(\eta_v);\eta_v)}{\sqrt{- f_2''(\alpha_2(\eta_v);\eta_v)}} e^{\frac{n}{4\pi \sqrt{-1}} f_2(\alpha_2(\eta_v);\eta_v)} \Big(1 + O\Big(\frac{1}{n}\Big) \Big)$$
\item If $p_2=1$, then
$$ \Sigma_{2} = \sqrt{2n}  \Big(1 +  (\sqrt{-1})^n e^{\frac{A_2}{2}} \Big) \frac{g_2(\pi+\alpha_2(\eta_v);\eta_v)}{\sqrt{- f_2''(\pi+\alpha_2(\eta_v);\eta_v)}} e^{\frac{n}{4\pi \sqrt{-1}} f_2(\pi+\alpha_2(\eta_v);\eta_v)}  \Big( 1 + O \Big( \frac{1}{n} \Big) \Big)$$
\end{enumerate}
\end{proposition}
\begin{proof}
(1) \& (3) follow by Proposition \ref{Sigmasum0}, \ref{Sigmasum}, \ref{leading1}, \ref{OtherFest1}, \ref{OtherFest2}, \ref{OtherFest3}. (2) \& (4) follow by Proposition \ref{Sigmasum0}, \ref{Sigmasum}, \ref{leading2}, \ref{OtherFest1'}, \ref{OtherFest2'}, \ref{OtherFest3'}.
\end{proof}

\begin{proposition}\label{lastpf} Write $\theta_v^{(n)} = \frac{4\pi m_v}{n}$. We have
\begin{align*}
| \mathrm{Trace}\text{ }( \Lambda_{\varphi, r, p_v}^q ) | = C(n)  \frac{e^{\frac{n}{4\pi} \mathrm{Vol}\big(M_{\varphi, \theta_v^{(n)}}\big)}}{\sqrt{ \pm\mathrm{Tor}(M_\varphi, \mu_v,  \rho_{\theta_v^{(n)}})}} \Big(1+O\Big(\frac{1}{n}\Big)\Big),
\end{align*}
where
$$C(n)
= \frac{|C_0(n) |}{\sqrt{2}\Big|D^q\Big(qe^{-\frac{A_1}{n}}\Big)\Big|^{\frac{1}{n}}\Big|D^q\Big(qe^{-\frac{A_2}{n}}\Big)\Big|^{\frac{1}{n}}}
$$ and
\begin{align*}
&C_0(n) \\
&=
\begin{cases}
e^{-\big(\frac{\hat{l}_0 \sqrt{-1}}{2} + \frac{V_1}{2\pi}\big) \alpha_1(\eta_v)
-\big(\frac{\hat{m}_0 \sqrt{-1}}{2} + \frac{V_2}{2\pi}\big) \alpha_2(\eta_v)} & \\
\times \Big(1+e^{-2\alpha_1(\eta_v)\sqrt{-1}} \Big)^{\frac{A_1}{4\pi \sqrt{-1}}}
\Big(1+e^{-2\alpha_2(\eta_v)\sqrt{-1}} \Big)^{\frac{A_2}{4\pi \sqrt{-1}}}  & \text{if $(p_1,p_2)=(0,0)$,}\\\\
e^{-\big(\frac{\hat{l}_0 \sqrt{-1}}{2} + \frac{V_1}{2\pi}\big) \alpha_1(\eta_v)
-\big(\frac{\hat{m}_0 \sqrt{-1}}{2} + \frac{V_2}{2\pi}\big)  (\alpha_2(\eta_v) + \pi)} & \\
\times \Big(1+e^{-2\alpha_1(\eta_v)\sqrt{-1}} \Big)^{\frac{A_1}{4\pi \sqrt{-1}}}
\Big(1+e^{-2\alpha_2(\eta_v)\sqrt{-1}} \Big)^{\frac{A_2}{4\pi \sqrt{-1}}} \Big(1 +  (\sqrt{-1})^n e^{\frac{A_2}{2}} \Big) & \text{if $(p_1,p_2)=(0,1)$,}\\\\
e^{-\big(\frac{\hat{l}_0 \sqrt{-1}}{2} + \frac{V_1}{2\pi}\big) (\alpha_1(\eta_v)+\pi)
-\big(\frac{\hat{m}_0 \sqrt{-1}}{2} + \frac{V_2}{2\pi}\big)  \alpha_2(\eta_v) } & \\
\times \Big(1+e^{-2\alpha_1(\eta_v)\sqrt{-1}} \Big)^{\frac{A_1}{4\pi \sqrt{-1}}}
\Big(1+e^{-2\alpha_2(\eta_v)\sqrt{-1}} \Big)^{\frac{A_2}{4\pi \sqrt{-1}}} \Big(1 +  (\sqrt{-1})^n e^{\frac{A_1}{2}} \Big) & \text{if $(p_1,p_2)=(1,0)$,}\\\\
e^{-\big(\frac{\hat{l}_0 \sqrt{-1}}{2} + \frac{V_1}{2\pi}\big) (\alpha_1(\eta_v)+\pi)
-\big(\frac{\hat{m}_0 \sqrt{-1}}{2} + \frac{V_2}{2\pi}\big)  (\alpha_2(\eta_v)+\pi) } & \\
\times \Big(1+e^{-2\alpha_1(\eta_v)\sqrt{-1}} \Big)^{\frac{A_1}{4\pi \sqrt{-1}}}
\Big(1+e^{-2\alpha_2(\eta_v)\sqrt{-1}} \Big)^{\frac{A_2}{4\pi \sqrt{-1}}} &\\
\times \Big(1 +  (\sqrt{-1})^n e^{\frac{A_1}{2}} \Big)\Big(1 +  (\sqrt{-1})^n e^{\frac{A_2}{2}} \Big) & \text{if $(p_1,p_2)=(1,1)$.}\\\\
\end{cases}
\end{align*}
Moreover, there exists $A,B>0$ such that $A<C(n)<B$.
\end{proposition}
\begin{proof}
Note that $f_{1}''(\pi+\alpha_1(\eta_v);\eta_v)=f_{1}''(\alpha_1(\eta_v);\eta_v)$ and $f_{2}''(\pi+\alpha_1(\eta_v);\eta_v)=f_{2}''(\alpha_1(\eta_v);\eta_v)$.
The asymptotic expansion formula follows from (\ref{Expfor}), Proposition \ref{Sigmasum0}, \ref{imgivevol}, \ref{Hess} and \ref{asym4}. 

To see the bound for $C_n$, by Lemma \ref{geotricone}, since $\mathrm{Im}(z_i(\eta_v))>0$, we have
$$ 1+e^{-2\alpha_i(\eta_v)\sqrt{-1}}  = 1-z_i (\eta_v)\neq 0$$
for $i=1,2$. Besides, if $1 +  (\sqrt{-1})^n e^{\frac{A_i}{2}} =0$ for some $i=1,2$, then we have $a_i=e^{A_i} = -1$, which contradicts with (\ref{period1}) and (\ref{period2}) that $b_1,b_2 \neq 0$. Together with Proposition \ref{asymDqu}, there exists $A, B>0$ such that 
$A<C(n)<B$.
\end{proof}

\noindent
Tushar Pandey\\
Department of Mathematics\\  Texas A\&M University\\
College Station, TX 77843, USA\\
(tusharp@tamu.edu)\\

\noindent
Ka Ho Wong\\
Department of Mathematics\\  Yale University\\
 New Haven, CT 06511, USA\\
(kaho.wong@yale.edu)
\\

\end{document}